\documentclass[runningheads,a4paper]{llncs}

\usepackage{amssymb}
\usepackage{amsmath}
\usepackage{mathtools}
\usepackage{graphicx}
\usepackage[colorlinks=true, linkcolor=blue]{hyperref}

\newcommand{\keywords}[1]{\par\addvspace\baselineskip
\noindent\keywordname\enspace\ignorespaces#1}

\newcounter{observation}[section]
\newenvironment{observation}[1][]{\refstepcounter{observation}\par\medskip
   \noindent \textbf{Observation~\theobservation. #1} \rmfamily}{\medskip}

\usepackage{comment}
\usepackage{xcolor}
\usepackage{tikz-cd}

\usepackage{mathabx}
\def\acts{\curvearrowright}

\newcommand{\Rho}{{\text{\sf P}}} 
\newcommand{\abrack}[1]{[\mkern-3mu[#1]\mkern-3mu]}

\begin{document}

\mainmatter  

\title{Notes on Tractor Calculi}

\titlerunning{Notes on Tractor Calculi}

%
%


\author{Jan Slov\'ak, Radek Such\'anek}

\authorrunning{Jan Slov\'ak, Radek Such\'anek}

\institute{Department of Mathematics and Statistics,
Faculty of Science, Masaryk University\\
Kotl\'a\v rsk\'a 267/2, 611 37 Brno, Czech Republic \\
}

%
%

\toctitle{Notes on Tractor Calculus}


\maketitle

\begin{abstract} 
These notes present elementary introduction to tractors based on classical 
examples, together with glimpses towards modern invariant
differential calculus related to vast class of Cartan geometries, the so
called parabolic geometries. 
\end{abstract}

This is a survey article based on the lectures given by the first author
at the Summer School Wis\l a 19, 19-29 August, 2019, captured by the second
author. The exposition aims at quick understanding of basic principles,
omitting many proofs or at least their details. The reader might find a lot of
further information in the cited sources throughout the text. In particular,
our approach has been heavily inspired by \cite{A}, while the general
background on Cartan geometries including the tractors can be found in
\cite{B}. 

The six sections of the article roughly correspond to the six lectures
(about 100 minutes each). We first introduce some elements of
the tractor calculus in quite general situation. Then we 
focus rather on the overall structure of the invariant linear differential 
operators and we do not present much of the tractor calculus itself. In this
sense, these lecture notes are complementary to \cite{A}, where the reader
should look for the genuine calculus.   

The audience was assumed to have basic knowledge
of differential geometry as well as some representation theory 
(Lie groups and algebras, their representations, principal and associated bundles,
connections, tensors, etc.). All this background can be found e.g. in
\cite{KMS} and \cite{B}.


 \section{Tracy Thomas' conformal tractors}

Let us start with a quick review of the two very well known geometries, the
Riemannian and the conformal Riemannian ones.

\subsection{Riemannian sphere} 
There are many ways how to view the standard sphere 
$$
S^n=\{ x \in \mathbb{R}^n\ |\ \|x \| = 1 \}
$$ 
as a homogeneous space. Perhaps the most common one is to consider the 
orthogonal group $ G = \operatorname{O}(n+1)$ which keeps $S^n$ invariant 
and its subgroup $H$ of the maps fixing a given point $o\in
S^n$, isomorphic to $\operatorname{O}(n)$.  
Clearly $S^n =
\operatorname{O}(n+1)/\operatorname{O}(n)$.  The Lie algebras
$\mathfrak{g}=\operatorname{Lie}G$ and $\mathfrak{h}=\operatorname{Lie}H$
enjoy the nice matrix $(1,n)$ block structure
$$ \mathfrak{g} = \begin{pmatrix} 0 & -v^{\operatorname{T}} \\
v & X \end{pmatrix}, \quad \mathfrak{h} = 
\begin{pmatrix} 0 & 0 \\ 0 & X \end{pmatrix}.
$$  
The tangent spaces are $T_xS^n = \{ y \in \mathbb{R}^{n+1} |\ \langle x,y
\rangle = 0 \} $ and the action of $G$ on $\mathbb R^{n+1}$, $x \mapsto Ax $,
preserves both $S^n$ and $TS^n$, i.e. $ y \mapsto Ay $ maps  $T_x S^n
\mapsto T_{Ax} S^n $. Moreover, they preserve the scalar products on the
tangent spaces and thus $S^n $ enjoys
$\operatorname{O}(n+1)$ as isometries of the natural structure of a
Riemannian manifold $(S^n,g)$.
 
\begin{observation}
There are no other isometries of $S^n$ apart from $\operatorname{O}(n+1)$. 
\end{observation}

The standard way to see the above observation holds true is the following. 
Consider a unit vector $e_1 \in \mathbb{R}^{n+1}$ and an isometry $\phi \in
\operatorname{Isom}(S^n, g)$.  Then $\phi (e_1) \in S^n$ and $\phi (e_1) = A
(e_1) $ for some $A \in \operatorname{O(n+1)}$.  Moreover, elements of the
form $A^{-1} \circ \phi $ are in the isotropy group of $e_1$.  As well known
the Riemannian isometries are (on connected components) uniquely determined
by their differential in one point.  Thus the latter map coincides with its
differential at $e_1$ and we are finished.

Another possibility is based on the Maurer-Cartan form. It is more
complicated, but much more conceptual.

Consider the principal $H$-bundle $G \xrightarrow{p} G/H$ over $S^n \cong
G/H$ equipped with the Maurer-Cartan form $\omega \in \Omega^1 \left( G,
\mathfrak{g} \right)$.  Since $\mathfrak{g}$ splits as $\mathfrak{g} =
\mathfrak{h} \oplus \mathfrak{n} $ (as $\mathfrak h$-module), 
the Maurer-Cartan form also splits as
$\omega = \omega_{ \mathfrak{h}} \oplus \omega_{ \mathfrak{n}}$, where the
first part is the principal connection $\omega_{ \mathfrak{h}}$ on $G$ and
the second part is the soldering form $\omega_{\mathfrak{n}}$.  The
soldering form provides for all $A \in G$ isomorphisms
\[
 \omega_{ \mathfrak{n}} \colon T_A G / V_A G \xrightarrow{\cong} T_{p(A)} S^n \cong \mathbb{R}^n  \ , 
\]
 where $V_A G : = \operatorname{ker} \omega_{\mathfrak{n}}$ is the vertical
 subspace.  This means $\omega_{\mathfrak{n}}$ makes $G \to S^n$ into
 a reduction of the linear frame bundle $P^1 S^n$ to $H$.  

Now any isometry $\phi$ lifts to the
 level of frame bundles and can be restricted to $G$ and thus we have a
 lift $\Tilde{\phi} \colon G \xrightarrow{} G$ such that $\Tilde{\phi}^{*}
 \omega_{ \mathfrak{n}} = \omega_{ \mathfrak{n}} $.  Because $\omega_{
 \mathfrak{h}} $ is principal connection preserving the metric we also have
 $\Tilde{\phi}^{*} \omega_{ \mathfrak{h}} = \omega_{ \mathfrak{h}} $.  We
 see that $\Tilde{\phi}^{*} \omega = \omega$, i.e.  $\Tilde{\phi}$ preserves
 the Maurer-Cartan form.  

Notice that in this setting, $\omega_{\mathfrak h}$ must be the only torsion
free metric connection on $S^n$. Thus, we arrived at the canonical Cartan
connection on $S^n$ in the sense of the so called \textit{Cartan
geometry} as defined below. 

For any principal fiber bundle $\mathcal G$ with structure group $P$, we
shall write $r^g$ for the principal right action of elements in $P$ and
$\zeta_X$ means the fundamental vector field, $\zeta_X(u)=\frac
d{dt}{}_{|_0}r^{\operatorname{exp}tX}(u)$.

 \begin{definition} \label{CG} For a pair $H \subset G$ of a Lie group and its Lie
subgroup, a Cartan geometry is a principal $H$-bundle $p: \mathcal{P} \to M$
endowed with a $\mathfrak{g}$-valued $1$-form $\omega \in \Omega^1 (
\mathcal{P},\mathfrak{g} ) $ satisfying for all $h \in H, X \in
\mathfrak{h}, u \in \mathcal{P}$ the following three properties

 \begin{align}
      \operatorname{Ad}(h^{-1}) \circ \omega & = (r^h)^{*} \omega , \\ 
      \omega (\zeta_X (u) ) & = X , \\
      \omega (u)  \colon T_u \mathcal{P} & \xrightarrow{\cong} \mathfrak{g} .
 \end{align}
 \end{definition}
 
Now, our observation follows from a general result:

\begin{theorem}[Fundamental theorem of calculus]\label{FTC}
 Let $\omega_G$ be the Maurer-Cartan form of a Lie group $G$ with the Lie
 algebra $\mathfrak{g}$, $M$ a smooth manifold endowed with a $1$-form
 $\omega \in \Omega^1(M, \mathfrak{g})$.  Then for each $x \in M$ there is a
 neighborhood $ U \ni x$ and $ f \colon U \to G$ such that $f^{*}
 \omega_G=\omega$, if and only if  
\begin{equation}\label{FTC-condition}
\operatorname{d} \omega + \frac{1}{2} [\omega, \omega] = 0 . 
\end{equation} 
If $M$ is connected and $f_1, f_2 \colon M \to G$ with $f_1^{*} \omega_G
 = f_2^{*} \omega_G$ on $M$, then there exists a unique $c \in G$ such that 
$f_2 = c f_1 $ on $M$.  
\end{theorem}

Under the additional requirement that $Tf:T_xM\to T_{f(x)}G$ is a linear
isomorphism for each point $x$, the theorem shows that the local Lie group
structure is uniquely determined by the Maurer-Cartan form satisfying
\eqref{FTC-condition}.
  
The theorem is proved by building the graph of the mapping $f$, see
\cite[section 1.2.4]{B}. Notice, in dimension one the condition is empty and
so with the additive group $G=\mathbb R$ we obtain just the existence of
primitive functions up to a constant. If $G$ was the multiplicative group
$\mathbb R_+$, the theorem would show how the logarithmic derivatives prescribe the
functions, up to a constant multiple. 

In our case each isomorphism $\phi:S^n\to S^n$ lifts to the unique map
$f:G\to G$ satisfying $f^*\omega = \omega$. The Maurer-Cartan forms on all
Lie groups satisfy the condition in \eqref{FTC-condition} and thus, even locally,
$f$ can differ from the
identity map only by an element of $G$. 
 
 
\subsection{Conformal Riemannian sphere} 
 
A conformal Riemannian manifold $(M , [g] ) $ is a manifold $M$ with a
conformal class of metrics.  Two metrics $g$, $\Tilde g$ 
are representatives of the same
conformal class if they differ by some positive function, $\Tilde{g} = \Omega^2 g$, $
\Omega \in C^{\infty} (M)$. Conformal isometry is a diffeomorphism, whose
differentials at all points belong to the conformal orthogonal group
$\operatorname{CO}(n)$ for the given structures on the tangent spaces.

The conformal sphere is $(S^n , [g] ) $ where $[g]$ includes the standard
round metric.  Let us discuss the following question: \textit{What is the group of all conformal isomorphisms on $S^n$
making it into a homogeneous space $G/P$?}

\smallskip
\noindent\textbf{Option 1.} We can go the `brutal force' way. Take $\mathbb
R^n$ with the conformal class containing the Euclidean metric and write down
the PDEs for an arbitrary locally defined conformal isomorphism $\phi$,
i.e., we request the differentials of $\phi$ are in $\operatorname{CO}(n)$ at
all points.  There is the famous Liouville theorem saying that each such
$\phi$ is generated by the Euclidean motions and the sphere inversions.  An
elementary (but tricky) proof can be found in \cite[section 5.4]{S-Vienna}. 
In particular, if we compactify $\mathbb R^n$ by the one point at infinity, we
can extend all such local diffeomorphisms to globally defined conformal maps
on $S^n$.

Let us try to do it in a smart way. Consider $\mathbb R^{n+2}$ with the
pseudo-Euclidean metric $Q(x,x)=2x_0x_{n+1}+x_1^2+\dots+x_n^2$ of signature
$(n+1,1)$ and define $C$ to be the null-cone of this metric.  Now, we may identify the
sphere $S^n$ with the projectivization $\mathbb PC$ of this cone and write down
the action of all the latter maps in projective coordinates on $\mathbb P\mathbb
R^{n+2}$.  We can represent the null-vectors of the affine $\mathbb R^n\subset
S^n$ as $(1:x:-\frac12\|x\|^2)$, while the remaining infinite point in $S^n$
is $(0:0:1)$.  Now we may easily identify the above conformal maps as actions of
particular matrices in $\operatorname{O}(n+1,1)$ on the projectivized cone
$\mathbb PC$.

\begin{align}
\label{action-co}
\begin{pmatrix} 1 \\ x \\ -\frac12\|x\|^2\end{pmatrix} &\mapsto
\begin{pmatrix} a^{-1} & 0 & 0 \\ 0 & A & 0 \\
0 &
0& a \end{pmatrix} \begin{pmatrix} 1 \\ x \\ -\frac12\|x\|^2
\end{pmatrix} 
\\
\label{action-translations}
\begin{pmatrix} 1 \\ x \\ -\frac12\|x\|^2\end{pmatrix} &\mapsto
\begin{pmatrix} 1 & 0 & 0 \\ v & E & 0 \\
-\frac12\|v\|^2 &
-v^T & 1 \end{pmatrix} \begin{pmatrix} 1 \\ x \\ -\frac12\|x\|^2
\end{pmatrix} 
\\
\label{action-sphere-inversion}
\begin{pmatrix} 1 \\ x \\ -\frac12\|x\|^2\end{pmatrix} &\mapsto
\begin{pmatrix} 0 & 0 & -2 \\ 0 & E & 0 \\
-\frac12 & 0& 0 \end{pmatrix}
\begin{pmatrix} 1 \\ x \\ -\frac12\|x\|^2
\end{pmatrix} 
\end{align}

Notice in the last line that the sphere inversion
$\sigma\in\operatorname{O}(n+1,1)$ is in 
a different component than the unit,
while the nontrivial maps fixing the origin and having the identity as 
differential there are obtained by composing $\sigma\circ\tau_v\circ \sigma$,
with the translation $\tau_v\in \operatorname{O}(n+1,1)$ from
\eqref{action-translations}, cf.
\cite[section 5.10]{S-Vienna}.

\smallskip
\noindent\textbf{Option 2.} Similarly to the Riemannian case, we first
choose the right homogeneous space $S^n=G/P$ with
$G=\operatorname{O}(n+1,1)$, 
$P$ the isotropy group of one fixed point in $S^n$,
and show that $G$ is just the group of all conformal isomorphisms. 
Again, we can achieve that by building a reasonably normalized Cartan
geometry for each conformal Riemannian manifolds. Then the Maurer-Cartan
form $\omega_G$ of $G$ will be preserved by all conformal morphisms and thus the 
Theorem \ref{FTC} applies. 

We shall come back to such normalizations of Cartan geometries later in
the fifth lecture.
 
At the level of Lie algebras, $\mathfrak g=\operatorname{Lie}G$ decomposes
as $\mathfrak g = \mathfrak g_{-1} \oplus \mathfrak p$, where $\mathfrak
g_{-1}$ are the infinitesimal translations with matrices 
$$
\mathfrak g_{-1} = \biggl\{ \begin{pmatrix} 0 & 0 & 0 \\ v & 0 & 0 \\
0 &
-v^T & 0\end{pmatrix}  
\biggr\}
$$
while
\[ \mathfrak{p} = \mathfrak g_0 \oplus \mathfrak g_1 = \biggl\{ \underbrace{ \begin{pmatrix} -a & 0& 0\\0 & A & 0\\ 
0& 0& a \end{pmatrix} }_{\mathfrak{co}(n)}  \oplus \underbrace{
\begin{pmatrix} 0& w & 0\\ 0& 0 & -w^T \\ 0& 0& 0 \end{pmatrix}
}_{\mathbb{R}^n}\biggr\}.  
\]
Clearly, $\mathfrak g_1$ is a $\mathfrak p$-submodule (actually an ideal), 
while $\mathfrak g_0$
is identified with the $\mathfrak p$-module $\mathfrak p/\mathfrak g_1$ with the
trivial action of $\mathfrak g_1$. This is the well known decomposition of
$\mathfrak p$ into
the reductive quotient and the nilpotent submodule.

At the level of Lie groups, this corresponds to the splitting of the
isotropy group $P$ into the semi-direct product of $\operatorname{CO}(n)$
containing the conformal isomorphisms fixing the origin and determined by
their first derivatives, and the nilpotent normal subgroup $P_+\subset P$ of 
those conformal isomorphisms fixing the origin with trivial first differential 
and determined by the
second order derivatives. We shall also write $G_0=P/P_+$ and this
reductive group decomposes further into the semisimple part
$\operatorname{O}(n)$ and the center $\mathbb R$.

%
 
%
%
 
\subsection{Towards tractors}\label{tractors-on-sphere}

The conformal Riemannian structure on $S^n$ can be read off the standard
metric $Q$ on $\mathbb R^{n+2}$ as follows. Any choice of a non-zero 
section of the null-cone $C$ (e.g. we may choose one of the non-zero components
$C_+$ of $C$ and consider sections there), seen as line bundle over its
projectivization $S^n$, provides the
identification of the tangent bundle $T_pS^n$ with the quotient of 
$$
T_pC_+ 
= \{z\in \mathbb R^{n+2}\ |\ Q(z,p)=0\}=\langle p\rangle^\perp
$$ 
by the line $\langle p\rangle$ (notice $p$
is null, so this line is in the tangent space). Clearly $\langle
p\rangle^\perp/\langle p\rangle$ is linearly isomorphic to $T_pS^n$ and
since $p$ is null, $Q$ induces a positive definite metric on $T_pS^n$. If we
multiply $p$ by a constant $a\ne0$, then the induced metric will change by the
positive multiple $a^2$. By the very construction, this conformal structure
is invariant with respect to the natural action of $\operatorname{O}(n+1,1)$
on the cone $C$. Thus, we may also view $C_+$ as the square root of the line
bundle of the conformal metrics in this class.

Of course, there is no preferred affine connection on $S^n$ in this picture. 
But if we consider the flat affine connection $\nabla$ on $\mathbb R^{n+2}$, then we
can consider the parallel (constant) vector fields in the trivial vector
bundle $C\times \mathbb R^{n+2}$ along the null-lines in $C$ and view them as
fields in the trivial vector bundle $\mathcal TS^n = S^n\times \mathbb
R^{n+2}$.       

The slight problem with this point of view is that we should expect that the
fibers of $\mathcal TS^n$ split into the `vertical part' along the null-lines 
in $C$, the `tangent part' to $S^n$ and the complementary 1-dimensional part
in $\mathbb R^{n+2}$. While the vertical part is well defined, 
such a splitting clearly depends on the choice of the
identification of $S^n$ with a section of $C\to S^n$. Moreover, we
should hope to inherit an invariant connection from the flat connection
$\nabla$ on $\mathbb R^{n+2}$. Before answering these questions, we are
going to indicate a much simpler abstract description of such objects and we
come back to these functorial objects and constructions in the fourth lecture.  
 
 
 
Let $H \subset G$ be a Lie subgroup and $G\to G/H$ the corresponding 
Klein geometry. Notice that $G \to G/H$ is a principal $H$-bundle. 
Consider any linear representation $\mathbb{V}$ of $G$ and 
the associated bundle $\mathcal{V} = G \times_H \mathbb{V}$, i.e., the
classes of the equivalence relations on $G\times \mathbb V$ 
given by $ (u,v) \sim (u \cdot h, h^{-1} \cdot v)$.\footnote{Here $u \cdot h$ 
is multiplication in $G$  and $h^{-1} \cdot v$  is the left-action of 
$H$ or $G$ on $\mathbb{V}$ given by the chosen representation.} 

In particular, we may identify the class  
$\abrack{u,v}$ with the couple $(u \cdot H, u \cdot v)$. 
Indeed, taking another representative, we arrive at 
$$
((u\cdot h)\cdot H, u\cdot h\cdot (h^{-1}\cdot
v))=(u\cdot H,u\cdot v)
$$ and thus 
$\mathcal{V}$ is the trivial bundle on $M$ 
\[ 
\mathcal{V} = G/H \times \mathbb{V}. 
\]
Moreover, there is the Maurer-Cartan form $\omega$ on $G$. Extending
$G\to G/H$ to the principle $G$-bundle
$\Tilde{G} = G \times_H G\to G/H$, the form $\omega$ uniquely extends to a
principal connection form $\Tilde \omega$ on $\Tilde G$. Finally, we can
further identify $\mathcal{V}$ as the associated space 
$\mathcal{V} = \Tilde{G} \times_G
\mathbb{V}$.  Thus we see that there is the induced connection $\nabla$ 
on all such bundles $\mathcal{V}$. 


 \subsection{Tracy Thomas' tractors $\mathcal{T}$}
Now we come back to the conformal sphere and we apply the above abstract 
construction. Thus,
$G=\operatorname{O}(n+1,1)$, and $H=P\subset G$ is the isotropy group of the
fixed origin $(1:0:0)$, i.e., the Poincare subgroup in $G$ with the Lie
algebra $\mathfrak p$ as discussed above.
Further, we may take $\mathbb T= \mathbb
R^{n+2}$ with the standard action of $G$.

The final ingredients we need are the weights of line bundles or more
general tensor bundles on conformal
manifolds. Consider $\mathbb R[w]$ as the representation of $P$ such that $P_+$
and $\operatorname{O}(n)$ act trivially, while the central element
$\lambda=\operatorname{exp}(aE)\in G_0$
acts as $\lambda\cdot x=\operatorname{e}^{-a w}x$. Here $E$ is the so called grading element in
$\mathfrak g$, i.e.,
$$
\lambda = \operatorname{exp}\begin{pmatrix} a & 0 & 0\\ 0 & 0 &
0\\ 0 & 0 & -a\end{pmatrix} = \begin{pmatrix} \operatorname{e}^a & 0 & 0\\ 0 &
\operatorname{Id}_n & 
0\\ 0 & 0 & \operatorname{e}^{-a}\end{pmatrix} 
.$$  
Now, the line bundles of weights $w$ are defined as 
$$
\mathcal E[w] = G\times_P \mathbb R[w]
.$$
At the level of the infinitesimal action, the central element $aE\in\mathfrak
g_0$ will act as $a\cdot x = -wax$.
Notice the minus sign convention -- this is
because we want the line bundle of the conformal metrics to get the weight
two. 

Taking tensor products with line bundles, we arrive at the general weighted
bundles $\mathcal V[w] = \mathcal V\otimes \mathcal E[w]$ for any
representation $\mathbb V$ of $P$ or even $G$. 

Now, look at the action \eqref{action-co} on the components of $\mathbb T$. We
see immediately that the representation space $\mathbb T=\mathbb R^{n+2}$ splits as
$G_0$-module into 
$$
\mathbb T = \mathbb R[1] \oplus \mathbb R^n[-1]\oplus \mathbb R[-1],
$$
where the right ends are $P$-submodules. In particular, $\mathbb R[-1]$
is a $P$-submodule, while $\mathbb R[1]$ is the projecting component. Thus,
the trivial bundle $\mathcal T$ splits (once a section
of $C_+$ and thus one of the metrics in the class are fixed): 
\[ \mathcal{T} =
\mathcal{E}[1] \oplus \underbracket{TS^n [-1] \oplus
\overbracket{\mathcal{E} [-1]}^{VC_+ }}_{TC_+} .
\]

As mentioned above, the bundle $\mathcal T$ comes equipped with the canonical
metric induced by $Q$ and $TC_+=(VC_+)^\perp$. Thus, there is the positive
definite metric $\mathbf g:TS^n[-1]\times TS^n[-1]\to \mathcal E$, i.e.
$\mathbf g$ is a section of $S^2(T^*S^n)[2]$. This is the conformal class of
metrics on $S^n$ viewed as the section of a weighted metric bundle and it
allows us to raise and lower tensor indices of arbitrary tensors exactly as 
in the Riemannian case, but at the expense of adding or subtracting the weight 2.
For example, we may write $\mathcal T=\mathcal E[1]\oplus
T^*S^n[1]\oplus \mathcal E[-1]$.

Finally,
any section $\sigma$ of the projecting part $\mathcal E[1]$ provides the
Riemannian metric $g=\sigma^{-2}\mathbf g$.

\section{Conformal to Einstein and the tractor connection}

Tracy Thomas came across his conformal tractors in
\cite{F}, when constructing basic invariants of conformal geometry 
via a linear connection on a suitable
vector bundle (instead of building an absolute parallelism in the Cartan's
approach). He succeeded in finding the simplest of such vector bundles,
together with an invariant linear connection. He also worked out the
necessary transformation properties based on the so called Schouten tensor. 

All these objects were reinvented in \cite{G} and here the authors also
discussed the following question:
\textit{Given a conformal class $[g]$ on a manifold, 
is there a representative of the class which is an Einstein metric?} We
shall follow this development and thus we shall find the Thomas' tractors
when prolonging a conformally invariant geometric PDE 
(also following \cite{A}). 

In the sequel, we shall use the abstract index formalism. Moreover we shall
mostly not distinguish between the bundles $\mathcal VM$ and the spaces of
their sections $\Gamma(\mathcal VM)$. Thus, we shall talk about vector fields
in $\mathcal E^a$ or one-forms in $\mathcal E_a$. Similarly, $\eta_{ab}$ is
either a two-form in $\mathcal E_{ab}$ or a $\mathcal E_a$-valued one-form.
As usual, repeated indices at different positions (lower versus upper) mean
the relevant trace.

\subsection{The Einstein scales}
Recall that the curvature $R_{ab}{}^c{}_d$ of the unique torsion-free metric
connection $\nabla$ decomposes into the trace-free Weyl tensor $W_{ab}{}^c{}_d$
and the Ricci tensor $R_{ab}$. We shall see later why a trace-adjusted
version of $R_{ab}$, the so called Schouten tensor,
$$
P_{ab} = \frac{1}{n-2}\left( R_{ab} - \frac{1}{2(n-1)} R g_{ab} \right)
$$
where $R=g^{ab}R_{ab}$ is the scalar curvature, is very useful. (Notice, 
here we use the opposite sign convention for the
Schouten tensor $P$ than
in \cite{B}, i.e. it is the same as in \cite{A}.)

Of course, we should believe there is an overdetermined distinguished 
PDE system on the scales, i.e. the choices of the metrics in the class, whose
solutions correspond to the Einstein scales. We can write down all such PDEs
with the help of any of the metrics in the class and, as a matter of fact,
the equation must be independent of our choice, i.e. \emph{conformally
invariant}. A straightforward check reveals
(we shall come to such techniques later) that the following equation on (the
square roots of) the scales
$\sigma$ in $\mathcal E[1]$ is invariant 
 \begin{align}\label{Einstein condition}
     \boxed{\nabla_{(a} \nabla_{b)_0} \sigma + P_{(ab)_0} \sigma = 0  }
 \end{align}
Now, if $\sigma$ is a nowhere zero solution, we may write the equation using 
the metric connection corresponding to $\sigma$ and thus, $\sigma$ is 
parallel and we arrive at $P_{(ab)_0}=0$. This is
exactly the condition to be Einstein, i.e., the trace-free part of Ricci
vanishes and $\sigma^{-2}\mathbf g_{ab}$ is Einstein.  

We are going to apply the classical method of prolongation of overdetermined
systems of PDEs to show that solutions of \eqref{Einstein condition} 
are equivalent to \textit{parallel tractors} in $\mathcal T$.

 
First, we add trace part $\rho\mathbf g$ to the equation \eqref{Einstein condition}, 
i.e, $\rho$ is a new $(-1)$-weighted quantity $\rho \in \mathcal{E} [-1]$
and the new equation becomes  
 \begin{equation}
      \nabla_{(a} \nabla_{b)} \sigma + P_{(ab)} \sigma + \mathbf{g}_{ab} \rho = 0  . 
 \end{equation} 
Moreover, we know that the Ricci curvature is symmetric for all scales, i.e.
the Levi-Civita connections of the metrics in the class, and 
thus the Schouten tensor is symmetric too. Finally, the antisymmetric part
of the second order derivative is given by the action of the curvature
$R_{ab}{}^c{}_d$ as a 2-form valued in the Lie algebra $\mathfrak{so}(n,
\mathbb R)$ and these values have no central component to act on the
densities $\mathcal E[w]$. Thus our equation becomes 
\begin{equation} \label{complete-eq}
\nabla_a \nabla_b \sigma + P_{ab} \sigma + \mathbf{g}_{ab} \rho = 0 .
\end{equation}
 

In the next step, we give the derivative $\nabla_a\sigma $ the new name
$\mu_a=\nabla_a \sigma \in \mathcal E_a[1]$.\footnote{We know that 
$\mu_a$ must be of weight $1$ because covariant differentiation does 
not alter weights and $\sigma$ is already of weight $1$.} 
Thus the latter equation \eqref{complete-eq} can be rewritten as the system
of two first order equations 
\begin{equation} \label{prolonged equation}    
\boxed{ 
\begin{aligned}
\nabla_a\sigma - \mu_a &= 0
\\  
\nabla_a\mu_b + P_{ab} \sigma + \mathbf{g}_{ab} \rho &= 0 
\end{aligned}
}
\end{equation}

This system is not yet closed since there is still the uncoupled 
variable $\rho$. Thus we have to prolong the system and we 
need some computational preparation first.

Recall the invariant conformal metric $\mathbf{g}$ is covariantly constant
in all scales and differentiate \eqref{complete-eq}:
\begin{equation}\label{differential consequence 1}
\nabla_a\nabla_b\nabla_c\sigma + \mathbf{g}_{bc} \nabla_a\rho + 
(\nabla_aP_{bc}) \sigma + P_{bc} \nabla_a\sigma   = 0 \ . 
\end{equation}
Contract \eqref{differential consequence 1} by hitting it with 
$\mathbf{g}^{ab}$ and $\mathbf{g}^{bc}$, respectively: 
\begin{align}\label{a)}
     \Delta (\nabla_c\sigma) + \nabla_c\rho + \nabla^a P_{ac} \sigma 
+ P^a{}_{c} \nabla_a\sigma &= 0 , 
\\ \label{b)}
     \nabla_a( \Delta \sigma ) + n \nabla_a\rho + 
\nabla_a P \sigma + P \nabla_a\sigma &= 0 ,
\end{align}
where $P$ is the trace of $P_{ab}$. Next, contracting the Bianchi
identity and some straightforward computations lead to 
\begin{align}\label{i)}
     \nabla^aP_{ac} & = \nabla_c P
\\ \label{ii)}
     [ \nabla_c, \Delta ] &= R_{cb}{}^b{}_d \nabla^d \ .
 \end{align}
 Subtracting \eqref{a)} from \eqref{b)} and 
using \eqref{i)} and \eqref{ii)} we arrive at
\begin{equation}\label{midstep1}
 (n-1)\nabla_c\rho + P \nabla_c\sigma - P^a{}_{c} \nabla_a\sigma +  
R_{cb}{}^b{}_d \nabla^d \sigma = 0 .
 \end{equation}
Further notice 
 \[ 
R_{cb}{}^{b}{}_{d} \nabla^d \sigma = - R_{ca} \nabla^a  \sigma = 
(2-n)P_c{}^{a} \nabla_a\sigma- \nabla_cP \sigma 
\]
 which together with \eqref{midstep1} yields, up to the constant factor $n-1$ 
 \begin{align}\label{differential consequence 2}
     \boxed{\nabla_c\rho - P_c{}^{a} \mu_a  = 0 }
 \end{align}
and our system of equations closes up. Summarizing, the Einstein scales
correspond to nowhere zero solutions of our system of three first order
equations coupling $\sigma$, $\mu_a$, and $\rho$ and all this should be
understood in terms of conformally invariant objects and operations.

Indeed, this is the content of the following theorem. For now, we formulate
it only for the solutions to our equations on the sphere $S^n$, although our
discussion on the equations has concerned general conformal 
Riemannian manifolds.

 
\begin{theorem}\label{Thomas-formula}
 Let $ \mathcal{T} =  \mathcal{E}[1] \oplus T^*S^n [1] \oplus \mathcal{E}[-1]$
be the bundle of the Thomas' tractors on the conformal sphere. 
Define the following operator on $\mathcal{T}$
\begin{equation}\label{standard_tractor_connection}
     \nabla^{\mathcal{T}}_a \begin{pmatrix} \sigma \\ \mu_a \\ \rho 
\end{pmatrix} = 
     \begin{pmatrix} \nabla_a\sigma - \mu_a \\ 
     \nabla_a\mu_{b} + \mathbf{g}_{ab} \rho + P_{ab} \sigma \\
     \nabla_a\rho - P_{ab} \mu^b  \end{pmatrix} ,
 \end{equation}
 where $\nabla_a$ on the right-hand side refers to the Levi-Civita
 connection of the metric $\sigma^{-2}\mathbf{g}$. 

The operator $\nabla^{\mathcal{T}}$ is a linear connection on $\mathcal{T}$ 
which is conformally invariant. 
Moreover, solutions to \eqref{Einstein condition} are in bijective 
correspondence with parallel tractors, 
i.e. with sections $t \in \Gamma (\mathcal{T}) $ such that $\nabla^{\mathcal{T}}
t = 0$.
\end{theorem}

Notice that on the sphere, $\mathcal T$ is the trivial bundle $S^n\times
\mathbb T$ and we shall see soon that $\nabla^{\mathcal T}$ is the flat
connection there which we mentioned earlier. So we also postpone the proof
of this theorem. (Actually, we see immediately that this is a connection,
but we should check its curvature and, in particular, how it depends on the
choice of the fixed metric.)

Obviously, all the parallel
tractors are determined uniquely by their values in $\mathbb T$ in the origin. 
In particular, we managed to compute all Einstein metrics on the
conformal sphere.

\subsection{Conformal invariance}

What do we really mean when saying that objects or operations are \emph{conformally
invariant}?

The intuitively obvious answer should be that they are independent of our choice
of the metric in the conformal class. So we should start to look how the
covariant derivative changes if we change the scale. Consider the change
of our metric by taking $\hat g = \Omega^2 g$ with a positive smooth
function $\Omega$, and write $\Upsilon_a = \Omega^{-1}\nabla_a\Omega$.

 \begin{lemma}\label{lemma-trans}
Let $\hat\nabla$ be the Levi-Civita connection for the rescaled metric $\hat g$. 
Then for all $v\in\mathcal E^a$, $\alpha\in\mathcal E_a$, $\rho\in\mathcal
E[w]$
 \begin{align}
     \hat{\nabla}_a v^b & = \nabla_a v^b + \Upsilon_a v^b - \Upsilon^b v_a +
\Upsilon^c v_c \delta^b_a  \\ 
     \hat{\nabla}_a \alpha_b & = \nabla_a \alpha_b - \Upsilon_a\alpha_b - 
\Upsilon_b \alpha_a + \Upsilon^c \alpha_c g_{ab}
\\
\hat{\nabla}_a\rho &= \nabla_a\rho + w\Upsilon_a.
 \end{align}
 \end{lemma}

 \begin{proof}
Recall the Christoffel symbols of Levi-Civita connection are expressed 
in any coordinates via the derivative of the metric coefficients 
(we write $\nabla_i$ for the partial derivatives here)
 \begin{equation}\label{Christoffels}
     \Gamma^i{}_{jk} = \frac{1}{2} g^{i\ell}( \nabla_k g_{\ell j} +  
\nabla_jg_{\ell k}  -  \nabla_{\ell}g_{jk} ) \ .
 \end{equation}
Conformal rescaling of the metric $g \mapsto \hat{g} =  \Omega^2 g $ 
affects all other objects derived from metric, e.g. the new inverse metric 
is $\hat{g}^{-1} = \Omega^{-2} g $. Thus, the Christoffels 
\eqref{Christoffels} change 
\begin{equation}
\begin{aligned} \hat{\Gamma}^i{}_{jk} &= \Gamma^i{}_{jk} + \frac{1}{\Omega} 
\left( \delta^i_j \nabla_k\Omega + \delta^i_k \nabla_j\Omega - g_{jk} 
\nabla^i \Omega\right)
\\
&= \Gamma^i{}_{jk} + \delta^i_j\Upsilon_k +
\delta^i_k\Upsilon_j - g_{jk}\Upsilon^i. 
\end{aligned}
 \end{equation}

Now recall, the covariant derivative is in coordinates given as the
directional derivative modified by the action of the Christoffels viewed as
$\mathfrak o(n)$-valued one-form. Thus the latter formula provides exactly
the three formulae in the statement.
\end{proof}

The formulae from the lemma allow to compute easily the changes of conformal
derivatives on all weighted tensor bundles. 

For example, considering possible first order operators on weighted forms 
$\mathcal E_a[w]$, we get
$$
\hat\nabla_a\alpha_b = \nabla_a \alpha_b + (w-1) \Upsilon_a\alpha_b - 
\Upsilon_b \alpha_a + \Upsilon^c \alpha_c g_{ab} ,
$$
and we immediately see that the antisymmetric part $\nabla_{[a}\alpha_{b]}$ is
invariant for the weight $w=0$ (this is the exterior differential on
one-forms), the trace-free part of the symmetrization
$\nabla_{(a}\alpha_{b)_0}$ is invariant for $w=2$ (we may view this as
the operator on the vector fields in $\mathcal E^a=\mathcal E_a[2]$ and the
kernel describes the conformal Killing vector fields),
and finally the trace $\nabla^a\alpha_a$ is invariant for $w=2-n$ 
(this is the divergence
of vector fields with weight $-n$).

Let us look at the geometric objects next. 
On the conformal sphere $S^n$, the category of natural objects was defined
in \ref{tractors-on-sphere} -- those are the homogeneous bundles
$G\times_P \mathbb V$ corresponding to any representation of $P$. If the
representation comes from a representation of $G_0=\operatorname{CO}(n)$,
extended by the trivial representation of $P_+=\operatorname{exp}\mathfrak
g_1$, the corresponding bundles extend to all conformal Riemannian
manifolds. Indeed, since general conformal Riemannian manifolds are 
given as reduction of
the linear frame bundles to the structure group $G_0$, such bundles 
are well defined on all of them. 

If we deal with more general $P$-representations, then we arrive at sums of
the latter bundles as soon as we fix a metric $g$ in the conformal class,
but the components are not given invariantly. We shall explain the general
procedure in the next lecture and, in particular, we shall see 
how this behavior extends
and defines such bundles on all conformal Riemannian manifolds. 
For now, just believe that in the case of the
Thomas' tractor bundles we face the following transformation rule
\begin{equation}\label{TTchange}
\begin{pmatrix}
\hat\sigma \\ \hat\mu_a \\ \hat \rho
\end{pmatrix}
=
\begin{pmatrix}
\sigma \\ \mu_a + \sigma\Upsilon_a \\ \rho - \Upsilon^c\mu_c -\frac12
\Upsilon^c\Upsilon_c\sigma
\end{pmatrix}
\end{equation}

Of course, a straightforward (and really tedious) 
computation can reveal that considering the formulae 
for the transformations of covariant derivatives from
the above Lemma and \eqref{TTchange}, the
linear tractor connection $\nabla^{\mathcal T}_a$ is a well defined
conformally invariant operator on the tractors. Fortunately, we do not have
to check this the pedestrian way and can wait for general
reasons.

 
\section{Parabolic geometries}
 
 We met the general Cartan geometries with the model $G/H$ 
in the Definition \ref{CG}. If the Lie group $G$ is semisimple and the
subgroup $H$ is a parabolic subgroup in $G$, we talk about the
\emph{parabolic geometries}. This class of Cartan geometries includes many
very important examples and provides a unified theory for all of them. In
this lecture we shall introduce some basic features and clarify many
phenomena in the conformal case on the way. Detailed exposition of the
background, including the necessary representation theory is available in
\cite{B}.

In general, the definition of the parabolic subgroups is a little subtle.
For us, the simplest approach is via their Lie algebras. The parabolic
ones are those which contain a Borel subalgebra and  
the choices of parabolic subalgebras $\mathfrak p\subset \mathfrak g$
correspond to graded decompositions of the semisimple Lie algebras 
$$
\mathfrak{g} = \mathfrak{g}_{-k} \oplus \dots \oplus \mathfrak{g}_k.
$$ 
This means that Lie brackets respect the grading, $[\mathfrak g_i,\mathfrak
g_j]\subset \mathfrak g_{i+j}$, and $\mathfrak p=\mathfrak g_0 \oplus
\mathfrak p_+=\mathfrak g_0 \oplus \mathfrak g_1\oplus\dots\oplus \mathfrak g_k$
is the decomposition into the reductive quotient $\mathfrak g_0$ and
nilpotent subalgebra $\mathfrak p_+$.
Moreover, there always is the unique \emph{grading element} $E$ in the
center of $\mathfrak g_0$ with the
property $[E,X]=jX$ for all $X\in\mathfrak g_j$.

The closed Lie subgroups $P\subset G$ are called parabolic if their algebras
$\mathfrak p=\operatorname{Lie}P$ are parabolic.

If $G$ is a complex semisimple Lie subgroup, then there is a nice geometric
description: $P\subset G$ is parabolic if and
only if $G/P$ is a compact manifold (and then it is a compact K\"ahler
projective variety),
see e.g., \cite[Section 1.2]{Zierau}. In
the real setting, the so called generalized flag varieties $G/P$ with
parabolic $P$ are always compact.

\subsection{$|1|$-graded parabolic geometries}
For the sake of simplicity, we shall restrict ourselves to the so called
$|1|$-graded cases here, i.e., $k=1$. 
Thus we shall deal with Lie groups with the algebras
 \begin{equation}\label{simple splitting of algebra}
     \mathfrak{g} = \mathfrak{g}_{-1} \oplus \underbracket{\mathfrak{g}_0 \oplus \mathfrak{g}_1 }_{\mathfrak{p}}
 \end{equation}
where $\mathfrak{p}$ refers to the parabolic subalgebra. 

Now, we consider Cartan geometries modelled on $G\to G/P$, i.e.  principal
$P$-bundles $\mathcal{G}\to M$ with Cartan
connections $\omega \in \Omega^1 (\mathcal{G}, \mathfrak{p})$. Such a
connection $\omega$ splits due to \eqref{simple splitting of algebra} as
 \[ \omega = \omega_{-1} \oplus \omega_0 \oplus \omega_1 . \] 
We shall further consider all reductions of the principal bundles $\mathcal
G$ to the structure group $G_0 = P/
\exp{\mathfrak{g}_1}$, i.e. we are interested in all equivariant mappings 
$$
\sigma \colon \mathcal{G}_0 =\mathcal G/\operatorname{exp}\mathfrak g_1\to 
\mathcal{G}
$$ 
with respect to the right
principal actions. The diagram below summarizes 
our situation (notice we are also fixing the subgroup $G_0$ in the semidirect
product $P= G_0\ltimes \operatorname{exp}\mathfrak g_1$, following the
splitting of the Lie algebra) 
 \[ 
 \begin{tikzcd}
 G_0\subset P \acts \mathcal{G} \ar[yshift=2pt]{r}{} & \mathcal{G}_0 \ar[yshift=-2pt]{l}{\sigma} \ar{r}{} & M .
 \end{tikzcd}
 \]

The Cartan connection $\omega$ allows us to identify the cotangent bundle $T^*M$
with $\mathcal G\times_P \mathfrak g_{1}$, where the action of $P_+$ is
trivial. Similarly $TM\simeq \mathcal G\times_P \mathfrak g/\mathfrak p$,
where $\mathfrak g/\mathfrak p\simeq \mathfrak g_{-1}$, again with trivial
action of $P_+$. The duality is provided by the Killing form on $\mathfrak
g$. 

Recall that all sections $\phi$ of associated bundles $\mathcal G\times_P\mathbb
V$ are identified with
equivariant functions $f:\mathcal G\to\mathbb V$, i.e. 
$\phi(x) = \abrack{u(x),f(u(x))}$
and so $$
f(u\cdot p) = p^{-1}\cdot f(u).
$$ 

Once we restrict the structure group to the reductive part of $P$, the 
pullback of the Cartan connection along $\sigma$ splits into the so called 
soldering form valued in $\mathfrak g_{-1}$, principal connection form valued
in $\mathfrak g_0$, and the one-form valued one-form 
$\operatorname{\operatorname{\Rho}}$ (which we shall see is the general
analog of the Schouten tensor $P_{ab}$ from conformal geometry) 
\[ 
\sigma^* \omega = \theta \oplus \sigma^* \omega_0 \oplus \operatorname{\Rho} . 
\] 

We can also take the other way round -- since $P_+$ is contractible (as the
exponential image of a nilpotent algebra), we may start 
with $\mathcal G=\mathcal G_0\times P_+$, fix one of the (reasonably
normalized) pullbacks $\sigma^*\omega_0$ and use some
suitable $\Rho$ to define the Cartan connection on the entire $\mathcal G$.
We shall see later, the Schouten tensor (with the opposite sign, see the
comment in the beginning of the second lecture) is the right choice to get the
normalized Cartan connection in the case of conformal Riemannian
geometries, taking one of the Levi-Civita connections for $\sigma^*\omega_0$. 
But we shall stay at the level of general Cartan connections now, 
so $\Rho$ is just the relevant pullback.

Two such reductions differ by a one form $\Upsilon$, viewed as equivariant
function $\Upsilon:\mathcal G_0\to \mathfrak g_1$:
$$
\hat \sigma = \sigma\cdot \operatorname{exp}\Upsilon
.$$

\subsection{Natural bundles and Weyl connections}
For each representation $\mathbb{V}$ of $P$ there is the functorial
construction of the bundles $\mathcal V=\mathcal G\times_P\mathbb V$ 
and the morphisms of the Cartan
geometries act on them in the obvious way. 

Actually we do not need the
Cartan connection for this definition, but notice that the morphisms of the
principal bundles respecting the Cartan connections are rather rigid in the
following sense. 
If we fix their projection to the base manifolds, 
the freedom in covering them is described by the kernel $K$ of the 
homogeneous models, i.e. the subgroup of $P$ acting
trivially on $G/P$, see \cite[section 1.5.3]{B}. This means two such 
morphisms may differ only by right principal action of elements from $K$.
Usually $K$ is trivial or discrete.  

Consider now a representation $\mathbb V$ of $P$ and its decomposition as a
$G_0$-module. The action of the
grading element $E\in \mathfrak g_0$ provides the splitting 
 \[ 
\mathbb{V} = \mathbb{V}_0 \oplus \dots \oplus \mathbb{V}_k , 
\] 
where the action of $\mathfrak g_1$ moves elements from $\mathbb V_i$ to $\mathbb
V_{i+1}$. Clearly, any section $v$ of $\mathcal V$
decomposes into the components $v_i:\mathcal G_0\to \mathbb V_i$ 
as soon as we fix our reduction. 

Further, fixing our reduction $\sigma$ we have the affine connection
$\omega_\sigma =
\sigma^{*} (\omega_{\leq 0} )$ (which is a Cartan connection, i.e., an
absolute parallelism, on the linear
frame bundle obtained as the sum of the soldering form and connection form).
As well known, the corresponding covariant derivative is obtained via the
constant vector fields:
$$
\nabla^\sigma_{\xi} v (u) = \omega^{-1}_{\sigma} (X) \cdot v (u)
$$ 
where $X\in \mathfrak g_{-1}$ corresponds
to the vector $\xi$ in a frame $u\in \mathcal G_0$, i.e., we simply
differentiate a function in the direction of a vector 
(the horizontal lift of
$\xi$ to $\mathcal G_0$). Notice that this connection $\nabla^\sigma$ always
respects the decomposition of $\mathcal V$ given by the same reduction. We call
all these connections the \emph{Weyl connections} (and we obtain the genuine
Weyl connections in the conformal case with the Schouten tensor $-\Rho$, i.e.
all the torsion free connections preserving the conformal Riemannian
structure). 

Our next theorem says, how the splitting of $\mathbb V$, the covariant
derivative, and also the one-form $\Rho$ change if we change the reduction
$\sigma$. 

In order to formulate the results, let us introduce some further
conventions. Recall, tangent vectors $\xi\in T_xM$ can be identified with
right-equivariant functions $X$ on the frames over $x$ valued in $\mathfrak
g_{-1}=\mathfrak g/\mathfrak p$. This identification can be written down
with the help of the Cartan connection, $X=\omega_{-1}(u)(\tilde \xi)$ 
for any lift $\tilde \xi$ of $\xi$ to
$T_u\mathcal G$. By abuse of notation we shall write the same symbol 
$\xi$ for the
vector in $TM$ and the corresponding element $X$ in $\mathfrak g_{-1}$.
Similarly we shall deal with the one forms $\Upsilon$ represented by
elements in $\mathfrak g_1$, 
and also the endomorphisms of $TM$ represented
by elements in $\mathfrak g_0$. 

For instance,
$\operatorname{ad}\Upsilon(\xi)\cdot v$ means we take the Lie algebra valued 
functions $\Upsilon$ and $\xi$, take the Lie bracket of their values 
and act by the result
on the value of the function $v$ via the representation of $\mathfrak g_0$ 
in question. Of
course, we may use only such operations which ensure the necessary
equivariance (which is guaranteed when taking the adjoint action within the
Lie algebra).
 

\begin{theorem}\label{transT}
Consider $\hat{\sigma} = \sigma \cdot \exp{\Upsilon}$ and use the hat to
indicate all the transformed quantities. For every section
$v=v_0\oplus\dots\oplus v_k$ in the representation space of the
representation $\lambda:\mathfrak p\to \mathfrak{gl}(\mathbb V)$, 
and vector $\xi$ in the tangent bundle, 
\begin{equation}\label{p-action}    
\hat{v}_\ell = (\lambda(\operatorname{exp}(-\Upsilon))(v))_{\ell}=\sum_{i+j = \ell} \frac{(-1)^i}{i!} \lambda (\Upsilon)^i (v_j)
.
\end{equation}
If $\lambda$ is a completely reducible $P$-representation, then
\begin{equation}\label{conn-trans}
\hat{\nabla}_{\xi} v = \nabla_{\xi} v - \operatorname{ad} \Upsilon (\xi) 
\cdot v
.\end{equation} 
Finally, the one-form $\Rho$ transforms
\begin{equation} \label{Rho-trans}
\hat\Rho = \Rho (\xi) + \nabla_{\xi} \Upsilon + \frac{1}{2} 
(\operatorname{ad} \Upsilon )^2(\xi)
.\end{equation}
\end{theorem}

\begin{proof}
The formula \eqref{p-action} is just a direct consequence of our 
definitions and
reflects the fact that by changing the reduction $\sigma$, the equivariant function
$v:\mathcal G\to \mathbb V$ is restricted to another subset, shifted 
by the right action of $\operatorname{exp}(\Upsilon)$. Thus, the values have to get
corrected by the action of
$(\operatorname{exp}\Upsilon)^{-1}=\operatorname{exp}(-\Upsilon)$. The
formula then follows by collecting the terms with the right homogeneities.   

The transformation of the derivative is also not too difficult. Consider a
section $v:\mathcal G_0\to \mathbb V$ and  
recall $\nabla_\xi v$ is given with the help of any lift  
$\tilde\xi$ to $\mathcal G_0$:
\begin{equation}\label{Weyl-derivative}
\nabla_\xi v=\tilde\xi\cdot v (u) - \omega_0(T_u\sigma\cdot \tilde
\xi)\cdot v(u)
.\end{equation}
Writing $r^p=r(\ ,p)$ and $r_u=r(u,\ )$ for the right action,
\begin{equation}\label{trans-eq}
T_u\hat\sigma\cdot \tilde\xi = T_{\sigma(u)}r^{\operatorname{exp}\Upsilon(u)}
\cdot T_u\sigma\cdot \tilde \xi +
T_{\operatorname{exp}\Upsilon(u)}r_{\sigma(u)}\cdot 
T_u{\operatorname{exp}\Upsilon}\cdot \tilde \xi
.\end{equation}
The second term in \eqref{trans-eq} is vertical in $\mathcal G\to \mathcal G_0$ and thus 
$$
\omega_0(T_u\hat\sigma\cdot \tilde\xi) = 
\omega_0(T_{\sigma(u)}r^{\operatorname{exp}\Upsilon(u)}
\cdot T_u\sigma\cdot \tilde \xi )
.$$
By equivariancy of the Cartan connection $\omega$, this equals to the
$\mathfrak g_0$ component of
$\operatorname{Ad}((\operatorname{exp}\Upsilon(u))^{-1})(\omega(T_u\sigma\cdot
\tilde\xi))$. Now, notice $\omega_{-1}(T_u\sigma\cdot \tilde\xi)$ is exactly
the coordinate function representing the vector $\xi$. 
Thus, the only $\mathfrak g_0$ component of 
the latter expression is
$\operatorname{ad}(-\Upsilon(u))(\xi)$ and this has to act on $v$ in our
transformation formula.  

The transformation of the $\Rho$ tensor is also deduced from
\eqref{trans-eq}, but it is more technical and we refer to the detailed proof
in \cite[section 5.1.8]{B}.
\end{proof}

Similar formulae are available for general parabolic geometries and their
Weyl connections. Just the non-trivial gradings of $TM$ and $T^*M$ make them
much more complicated. The complete exposition can be read from
\cite[sections 5.1.5 through 5.1.9]{B}. 

Notice also that we are allowing all reductions $\sigma$. 
But some of them are nicer
than others -- we may reduce the structure group to the semisimple part
$G_0^{ss}$ of
$G_0$. These further reductions correspond to sections of the line bundle
$\mathcal L=\mathcal G_0/G_0^{ss}$, which can be viewed as the associated
bundle $\mathcal G_0\times_{G_0}\operatorname{exp}\{wE\}$ carrying the natural
structure of a principal bundle with structure group $\mathbb R_+$. 
This is
the line bundle of scales and its sections correspond to Weyl connections
inducing flat connections on $\mathcal L$. In the conformal case, these are
just the choices of
metrics in the conformal class. The induced connection on $\mathcal L$ has
got the antisymmetric part of $\Rho$ as its curvature and thus, 
we can recognize such more special 
reductions by the fact that for these the Rho-tensor is symmetric.

\subsection{Higher order derivatives}
Notice, in Theorem \ref{transT} we provided the formula for the change of the Weyl connections for completely reducible $P$-modules only. This is
because the formulae get very nasty for modules with nontrivial $\mathfrak
g_1$ actions. But even dealing with tensorial bundles, iterating the
derivatives always leads to such modules. 

In order to avoid at least part of these hassles, we should seek for better
linear connections related to our reductions $\sigma$ and the fixed Cartan
connection $\omega$. An obvious choice
seems to be the following one. Fixing a reduction $\sigma$ consider the
principal connections $\mathcal G$ with the connection form
$\gamma^\sigma\in\Omega^1(\mathcal G,\mathfrak p)$,
$$
\gamma^\sigma(\sigma(u)\cdot g)(\xi) = \omega_{\mathfrak
p}(\sigma(u))(Tr^{g^{-1}}\cdot
\xi)
$$
for all $u\in\mathcal G_0$, $\xi\in T_{\sigma(u)\cdot g}$, $g\in P_+$. In
words, we restrict the $\mathfrak p$-component of $\omega$ to the image of
$\sigma$ and extend it the unique way to a principal connection form.

Clearly, this connection form defines the associated linear connections on
all natural bundles, we call them the \emph{Rho-corrected Weyl
connections $\nabla^{\Rho}$}. They were perhaps first introduced in
\cite{PG-Adelaide} and exploited properly in \cite{Rho-corrected}.
In the case of conformal Riemannian structures, these concepts are closely
related to the so called W\"unsch's conformal calculus, cf. \cite{Wunsch}.

\begin{theorem}\label{Rho-corrected}
Consider natural bundle $\mathcal V = \mathcal G\times_P \mathbb V$ and a
reduction with the Weyl connection $\nabla$ and its Rho-corrected derivative
$\nabla^\Rho$. Then
\begin{align}
\label{formula-Rhocorrected}
\nabla^\Rho_\xi v &= \nabla_\xi v + \Rho(\xi)\cdot v
\\
\hat\nabla^\Rho_\xi v &= \nabla^\Rho_\xi v +
\sum_{i\ge1}\frac{(-1)^i}{i!}(\operatorname{ad}\Upsilon)^i(\xi)\cdot v
.\end{align}
\end{theorem}

\begin{proof}
In order to see the difference between $\nabla$ and $\nabla^\Rho$, we can
inspect the expression \eqref{Weyl-derivative} with the choice of the
horizontal vector field $\tilde \xi$ lifting $\xi$. Thus $\nabla_\xi v=
\tilde\xi\cdot v$. On the other hand, choosing the lift 
$T\sigma\cdot\tilde\xi$ on $\mathcal G$, we obtain 
$\omega_0(T\sigma\cdot\tilde\xi)=0$ and $\omega_1(T\sigma\cdot\tilde\xi)$
represents $\Rho(\xi)$. Equivariancy of $v$ then implies our formula
\eqref{formula-Rhocorrected} along the entire image of $\sigma$.

Let us now consider the horizontal lift $\tilde \xi$ of 
$\xi$ on $\mathcal G$ with
respect to $\gamma^\sigma$. Then $\nabla^\Rho_\xi v$ is represented by
$\tilde\xi\cdot v$, while $\tilde\xi\cdot v +
\gamma^{\hat\sigma}(\tilde\xi)\cdot v$ represents $\hat\nabla_\xi v$. By the
very definition, $\omega(\sigma(u))(\tilde \xi)\in \mathfrak g_{-1}$. Thus,
$$
\gamma^{\hat\sigma}(\sigma(u))(\tilde\xi) = \omega_{\mathfrak
p}(Tr^{\operatorname{exp}\Upsilon(u)}\cdot \tilde\xi(\sigma(u)))
,$$ 
which is just the component of
$\operatorname{Ad}((\operatorname{exp}\Upsilon(u))^{-1})
(\omega(\sigma(u))(\tilde\xi))$. Now, notice that 
$\omega(\sigma(u))(\tilde\xi)$
represents $\xi$ by values in $\mathfrak g_{-1}$ and the requested formula
follows.
\end{proof}

We should notice that the Weyl connections and the Rho
corrected ones coincide on bundles coming from representations with trivial
action of $P_+$. Of course, the transformation formulae coincide in this
case, too.

\subsection{A few examples} 

We shall go through a few homogeneous models and comment on the general
`curved' situations. In all cases the actual geometric structures are given
by the reductions of the linear frame bundles and the construction of the
right Cartan geometry is a separate issue. We shall come back to this in the
fifth lecture and work with the general choices of the Cartan connections
$\omega$ here.

\smallskip

\noindent\textbf{Conformal Riemannian geometry.}
The relevant Cartan geometry can be modelled by the choice  
$G = \operatorname{O}(n+1, 1,\mathbb{R})$ (there is some
freedom in the choice
of the group with the given graded Lie algebra $\mathfrak g$) and the parabolic subgroup $P$ as
we saw in detail in the first lecture. 
%

It is a simple exercise now to recover the formulae from Lemma
\ref{lemma-trans} by computing the brackets in the Lie algebra. In our
conventions using the coordinate functions instead of fields, we can rewrite
them as (notice $\alpha$ is valued in $\mathfrak g_1$, while $\eta$, $\xi$
have got values in $\mathfrak g_{-1}$, and $s$ sits in $\mathbb R[w]$) 
\begin{align}
\hat \nabla_\xi\eta &= \nabla_\xi\eta - [\Upsilon,\xi]\cdot \eta
\\
\hat\nabla_\xi\alpha &= \nabla_\xi\alpha -[\Upsilon,\xi]\cdot \alpha
\\
\hat\nabla_\xi s &= \nabla_\xi s - (-\Upsilon(\xi))w s 
\end{align}
where we have picked up just the central component of the bracket in the
last line, viewed as the multiple of the grading element.

Another, but still much more tedious exercise would be to check the
conformal invariance of the tractor connection on $\mathcal T$. We shall
develop much better tools for that in the next lecture.

We shall also enjoy much better tools to discuss second or higher order
operators. For example, considering second order operators on densities
$s\in\mathcal E[w]$, we may iterate the Rho-corrected
derivative to obtain
$$
\mathbf g^{ab}\nabla^\Rho_a\nabla^\Rho_b s = \nabla^a\nabla_a s - w \mathbf
g^{ab}\Rho_{ab} s
$$  
and check that this gets an invariant operator for $w=1-\frac n2$, which is
the famous conformally invariant Laplacian, the so called Yamabe operator
$$
Y:\mathcal E[1-\frac n2]\to \mathcal E[-1-\frac n2]
.$$

\smallskip
\noindent\textbf{Projective geometry.} The choice of the homogeneous 
model is obtained from the algebra of trace free real matrices 
$\mathfrak g = \mathfrak{sl}(n+1,\mathbb{R})$ 
with the grading
 \[ 
    \left(
        \begin{array}{c|c}
        \mathfrak{z} & \mathbb{R}^{n*} \\
        \hline
        \mathbb{R}^n & \mathfrak{gl} (n,\mathbb{R}) 
        \end{array}
    \right)
    \begin{matrix}
    1 \\ 
    n
    \end{matrix}
 \] 
Here $\mathfrak{z}=\Bbb R$ is the center, the grading element $E$
corresponds to
$\frac n{n+1}$ and $-\frac1{n+1} \operatorname{id}_{\mathbb R^n}$ on the diagonal. 
We may take $G=\operatorname{SL}(n+1,\mathbb R)$ and $P$ the subgroup of
block upper triangular matrices. The homogeneous model is then 
the real projective space $G/P = \mathbb{R} \mathbb{P}^n$. On the
homogeneous model, the Weyl connections transform as
$$
\hat \nabla_\xi\eta = \nabla_\xi\eta + \Upsilon(\eta)\xi + \Upsilon(\xi)\eta
$$
and so they clearly share the geodesics.

For general projective structures on manifolds $M$, the 
space of Weyl connections has to be chosen as a class of all affine
connections sharing geodesics with a given one and they transform then the
same way. We shall see, that projective geometries are rare exceptions of
parabolic geometries not given by a first order structure on the manifold.

The analog of the Thomas' tractors is the natural bundle corresponding to
the standard representation of $\operatorname{SL}(n+1,\mathbb R)$ on
$\mathbb T=\mathbb R^{n+1}$. The injecting part of $\mathcal T$ is the line 
bundle $\mathcal T^1$ with the action of the grading element by $\frac
n{n+1}$. The usual convention says this is the line bundle $\mathcal E[-1]$.
Then the projecting component is the weighted tangent bundle $TM[-1]$. 

\smallskip\textbf{Almost Grassmannian geometry.} This is essentially a
continuation of the previous example. We take $G=\operatorname{SL}(p,q)$ and
the splitting of the matrices into blocks of sizes $p$ and $q$, say 
$2\le p\le q$. Unlike the projective case, here the geometry is determined
by reducing the structure group of the tangent bundle to
$\mathbb R\times \operatorname{SL}(p,\mathbb
R)\times\operatorname{SL}(q,\mathbb R)$. This corresponds to identifying
the tangent bundle with the tensor product of the auxiliary bundles $\mathcal
V^*$ and $\mathcal W$ of dimensions $p$ and $q$, together with the
identification of their top degree forms $\Lambda^p\mathcal V\simeq
\Lambda^q\mathcal W^*$. 

Thus, we may use the abstract indices and write $\mathcal V=\mathcal
E^{A}$, $\mathcal W=\mathcal E^{A'}$. Then the tangent bundle is 
$\mathcal E^{A'}_A$ and the formula for the brackets in the Lie algebra says
$[[\Upsilon,\xi],\eta]^{A'}_A= -\xi^{A'}_B\Upsilon^B_{B'}\eta^{B'}_A - 
\xi^{B'}_A\Upsilon^B_{B'}\eta^{A'}_B$. 
The Weyl connections are tensor products of connections on $\mathcal V^*$
and $\mathcal W$ (but not all of them).
The right formula for the change of the Weyl connections is
$$
\hat \nabla^A_{A'}\eta^{B'}_B = \nabla^A_{A'}\eta^{B'}_B +
\delta^{B'}_{A'}\Upsilon^A_{C'}\eta^{C'}_B +
\delta^{A}_{B}\Upsilon^C_{A'}\eta^{B'}_C
.$$ 

The analog to the Thomas' tractors comes from the standard representation of
$G$ on $\mathbb T=\mathbb R^{p+q}=\mathcal V \oplus \mathcal W$. 
Thus, fixing a Weyl connection, we get
the tractors as couples $(v^A, w^{A'})$ with the transformation rules 
$$
\hat v^A = v^A-\Upsilon^A_{B'}w^{B'}, \quad \hat w^{A'} = w^{A'} 
.$$ 

Notice the special case $p=q=2$ which provides (the split real form of) 
the Penrose's spinor presentation of tangent bundle and the two-component
four-dimensional twistors $\mathcal T$. Indeed, $\mathfrak{so}(6,\mathbb
C)=\mathfrak{sl}(4,\mathbb C)$ and $\mathfrak{so}(4,\mathbb C)$ splits into
sum of two $\mathfrak{sl}(2,\mathbb C)$ components. Thus, up to the choice
of the right real form, the almost Grassmannian geometries with $p=q=2$
correspond to the four-dimensional conformal Riemannian geometries.

The twistor parallel transport (connection) is then given by the
formula
$$
(\nabla^{\mathcal T})^A_{A'}\begin{pmatrix} v^B \\ w^{B'} \end{pmatrix}
= \begin{pmatrix} \nabla^A_{A'}v^B + \Rho^{AB}_{A'C'}w^{C'} \\ 
\nabla^A_{A'} w^{B'} + \delta^{B'}_{A'}v^A \end{pmatrix}
$$
and we shall see that this is the right formula for the standard tractor
connection for the almost Grassmannian geometries in all dimensions.

\smallskip

The reader can find many further explicit examples in the last two chapters
of \cite{B}, including those with nontrivial gradings on $TM$.





 
 \section{Elements of tractor calculus}

In order to show how simple and general the basic functorial constructions
and objects are, we shall focus for a while on general Cartan geometries with 
Klein models $G\to G/H$ without any further assumptions. But we shall come
back to the parabolic and, in particular, conformal geometries in the end 
of this lecture.

\subsection{Natural bundles and tractors}
Let us come back to the functorial constructions on homogeneous spaces $G\to
G/H$ mentioned in the first lecture. As always, 
$\mathfrak h\subset\mathfrak g$ are the Lie algebras of $H$ and $G$.
%

For any Klein geometry $G/H$, there is the category of the homogeneous
vector bundles, where the objects are the associated bundles $\mathcal
V=G\times_H \mathbb V$.
All morphisms on $G/H$ are the
actions of elements of $G$ and these are mapped to the obvious actions on
$\mathcal V$. Further morphisms in this category are 
the linear mappings intertwining the actions of the elements of $G$. 

Clearly, there is the 
functor from the category of $H$-modules mapping the modules 
$\mathbb V$ to the
associated bundles $\mathcal V=G\times_H \mathbb V$, while any module
homomorphism $\phi:\mathbb V\to \mathbb W$ provides 
the morphisms $\abrack{u,v}\mapsto
\abrack{u,\phi(v)}$ between these bundles. 
 
The latter functorial construction extends obviously to the entire 
category $\mathcal{C}_{G/H}$ of all Cartan geometries modelled on 
$G/H$. The morphisms have to respect the Cartan connections $\omega$ on the
principal fiber bundles. 
 
In this setting, a natural bundle is a functor $\mathcal{V}
\colon \mathcal{C}_{G/H} \rightarrow \mathcal{VB}$ valued in the category of
vector bundles.  The functor sends
every Cartan geometry $(\mathcal{G} \rightarrow M, \omega) $ to the vector
bundle $\mathcal{V} M \rightarrow M$ over the same base (so it is a special
case of the so called gauge-natural bundles, see \cite{KMS}). Moreover,
$\mathcal{V}$ has the property that whenever there is a morphism between
objects of $\mathcal{C}_{G/H}$, $\Phi \colon (\mathcal{G} \rightarrow M,
\omega) \rightarrow (\tilde{\mathcal{G}} \rightarrow \tilde{M},
\tilde{\omega})$ covering $f \colon M \rightarrow \tilde{M}$, then there is
the corresponding vector bundle morphism $\mathcal{V} \Phi \colon \mathcal{V} M
\rightarrow \mathcal{V} \tilde{M} $ covering $f$.  This is just an explicit
description of the functoriality property with respect to the category of
Cartan geometries.
The main point is that each representation of $H$ produces such a functor
for all general Cartan geometries of the given type $G/H$. 

At the
same time, the Maurer-Cartan equation $\operatorname{d}\omega
+\frac12[\omega,\omega]=0$, valid on the homogeneous model, 
is no more true in general and we
obtain the definition of the \emph{curvature} $\kappa$ of the Cartan geometries $(\mathcal
G\to M, \omega)$ instead:
\begin{equation}\label{curvature}
\kappa = \operatorname{d}\omega + \frac12[\omega,\omega]
.\end{equation}

The fundamental Theorem \ref{FTC} immediately reveals that a general Cartan
geometry is locally isomorphic to its homogeneous model, if and only if
its curvature vanishes identically. 

We should also notice that there is the projective component of the curvature in
$\mathfrak g/\mathfrak h$ which we call the \emph{torsion}. Thus, the Cartan geometry
is \emph{torsion-free} if the values of its curvature $\kappa$ are in
$\mathfrak h$. We shall see later that the normalizations of Cartan
geometries consist in prescribing more complicated curvature restrictions,
which always depend on the algebraic features of the Klein models.

As already mentioned, we are interested in specific functors on Cartan geometries
$(\mathcal G,\omega)$ of the form $\mathcal{G}
\times_H - $, referring to the associated bundle construction given for each
fixed representation of $H$.  See \cite[section 1.5.5]{B} for a detailed discussion on the topic
of natural bundles on Cartan geometries.  Specializing to representations of $H$ which come as 
restrictions of representations of the whole group $G$ leads to the following
definition of \textit{tractor bundles} below.

Recall the sections $v$ of natural bundles $\mathcal V$ are identified with equivariant
functions $v:\mathcal G\to \mathbb V$, i.e. $v(u\cdot g) = g^{-1}\cdot
v(u)$. In particular, consider $\mathbb{V} =
 \mathfrak{g} / \mathfrak{h}=\mathbb R^n$ with the truncated adjoint action
of $H$ (i.e. the induced action on the quotient).
The Cartan connection $\omega$ allows us to identify every tangent vector
$\xi\in T_xM$ with the equivariant function $v:\mathcal G\to \mathbb V$,
$u\mapsto \omega(\tilde\xi(u))$ for an arbitrary lift $\tilde \xi$ of $\xi$. 
This is the identification of the tangent bundle $TM\simeq \mathcal VM$. (And it
completely justifies our earlier quite sloppy usage of elements in $\mathfrak g_{-1}$
instead of tangent vectors etc.)   

So in this way, the
 Cartan connection provides soldering of the tangent bundle, i.e. each
element $u\in \mathcal G$ in the fiber over $x\in M$ can be viewed as a
frame of $T_xM$. In general, different elements $u$ may represent the same
frame, depending on whether the truncated adjoint action of $H$ on
$\mathfrak g/\mathfrak h$ has got a non-trivial kernel. 

\begin{definition}\label{tractor bundle}
 The tractor bundles are natural vector
 bundles associated to the Cartan geometry $(\mathcal{G} \rightarrow M,
 \omega)$ of type $G \to G/H$, via restrictions of a representations of $G$ to
 the subgroup $H$.  

The unique principal connection form $\tilde \omega\in
\Omega^1(\tilde{\mathcal G}) \to \mathfrak g$ 
on the extended principal $G$-bundle
$\tilde{\mathcal G} = \mathcal G\times_H G$ extending the Cartan connection
$\omega$ on $\mathcal G$ induces the so called \emph{tractor connections}
$\nabla^{\mathcal V}$ on all tractor bundles $\mathcal VM$.
\end{definition}

Notice that $\tilde{\mathcal G}$ is indeed a $G$-principal fiber bundle with
the action of $G$ defined by the right multiplication on the standard fiber
$G$. Moreover, $u\mapsto \abrack{u,e}$ provides the canonical inclusion of the
principal fiber bundles $\mathcal G\subset \tilde{\mathcal G}$. 
The requested invariance of $\tilde\omega$, together with the
reproduction of the fundamental vector fields, define the values of 
$\tilde\omega$ completely from its restriction $\tilde \omega =\omega$ on
$T\mathcal G$.

In fact, we can equivalently define the tractor connections on the tractor
bundles directly (by specifying their special properties), instead of
referring to the Cartan connections on $\mathcal G$. This was also the
approach by Thomas in \cite{F}. The equivalence of such approaches for
$|1|$-graded parabolic geometries was noticed and exploited in \cite{E}. In
full generality, the construction, normalization and properties of tractor
connections were derived in
\cite{Cap-Gover} (see also \cite[Sections 1.5 and 3.1.22]{B}).  
 
 

 
\subsection{Adjoint tractors}
 
A prominent example of tractor bundles arises when
considering the $\operatorname{Ad}$ representation of the Lie group $G$ on
its Lie algebra $\mathfrak{g}$
and restricting it to $H$.  Applying the corresponding associated bundle
construction
$\mathcal{G} \times_H - $ on the following short exact
sequence of Lie algebras (with the obvious $\operatorname{Ad}$ actions)
 \[ 0 \to \mathfrak{h} \to \mathfrak{g} \to \mathfrak{g}/\mathfrak{h} \to 0
\]
we obtain  
 \begin{align}\label{sequence of bundles}
     0 \to \mathcal{G} \times_H \mathfrak{h} \to \mathcal{AM} \xrightarrow{\pi} TM  \to 0 \ ,
 \end{align}
where we have identified $TM \cong \mathcal{G} \times_H 
\mathfrak{g} / \mathfrak{h}$. The middle term 
$\mathcal{AM} : = \mathcal{G} \times_H \mathfrak{g} $ is called the 
\textit{adjoint tractor bundle}. 

%

Let us come back to the curvature \eqref{curvature} of the Cartan geometry
now. Clearly we may evaluate $\kappa$ on the so called \emph{constant vector
fields} $\omega^{-1}(X)$ for all $X\in\mathfrak g$. Consider $X\in\mathfrak
h$ and any $Y\in\mathfrak g$. Then $\omega^{-1}(X)$ is the fundamental
vector field $\zeta_X$ and
$\operatorname{d}\omega(\zeta_x,-)=i_{\zeta_X}\operatorname{d}\omega = \mathcal
L_{\zeta_X}\omega = -\operatorname{ad}(X)\circ \omega$, by the equivariancy
of $\omega$. Thus,
$$
\kappa(\omega^{-1}(X),\omega^{-1}(Y)) = \kappa(\omega^{-1}(X),\omega^{-1}(Y))
= -\operatorname{ad}(X)(Y) + [X,Y] = 0
.$$
We have concluded that, actually, the curvature is a horizontal 2-form which
can be represented by the equivariant \emph{curvature function} 
\begin{equation}\label{curvature-function}
\begin{aligned}
\kappa &: \mathcal G \to \Lambda^2(\mathfrak g/\mathfrak h)^*\otimes\mathfrak
g, 
\\ 
\kappa(X,Y)(u) &= 
-\omega([\omega^{-1}(X),\omega^{-1}(Y)](u)) +[X,Y].
\end{aligned} 
\end{equation}  
In particular, we understand that the curvature descends to a genuine 2-form
on the base manifold $M$ valued in the adjoint tractors, i.e. $\kappa \in
\Omega^2(M,\mathcal AM)$.
  

There is much more to say about the adjoint tractors, we shall summarize
several observations in the following two theorems (both were derived in
\cite{Cap-Gover}, see also \cite{B}).

\begin{theorem}
    \begin{enumerate}
        \item There is the (algebraic) Lie bracket 
$\{ \ ,\  \} \colon \mathcal{AM} \times \mathcal{AM} \to \mathcal{AM} $
inherited from the Lie bracket on $\mathfrak g$.
        \item The adjoint tractors are in bijective correspondence with the 
right-equivariant vector fields in $\mathcal X (\mathcal{G})^H$, and the
Lie bracket of vector fields on $\mathcal G$ equips $\mathcal AM$ 
with the differential Lie bracket $[\ ,\ ]$, which is compatible with the
Lie bracket on the tangent bundle $TM$, i.e. 
$\pi [ \zeta, \eta] = [ \pi \zeta, \pi  \eta] $.\footnote{Recall that $\pi \colon  \mathcal{AM} \to TM $ is the projection from sequence \eqref{sequence of
bundles}.}
        \item If $\mathcal V$ is a tractor bundle, then there is the natural
map $ \bullet  \colon \mathcal{AM} \times \mathcal{VM} \to 
\mathcal{VM}$, corresponding to the action of $\mathfrak g$ given by the 
$G$-representation $\mathbb{V}$. Moreover,  
$ \{ s_1, s_2\} \bullet t = s_1 \bullet s_2 \bullet t - s_2 \bullet s_1
\bullet t$.
        \item The bracket $\{ \ , \ \}$ and the actions $\bullet$ are 
parallel with respect to the tractor connections $\nabla^{\mathcal A}$,
$\nabla^{\mathcal V}$, i.e. for $s \in \mathcal{AM}$ and $ v \in \mathcal{VM} $
we know
        \begin{align*}
            \nabla^{\mathcal{A}}_{\xi} \{ s_1, s_2\} & = \{
\nabla^{\mathcal{A}}_{\xi} s_1, s_2\}  + \{ s_1, \nabla^{\mathcal{A}}_{\xi} s_2 \}, \\
            \nabla^{\mathcal{V}}_{\xi} ( s \bullet v ) & = (
\nabla^{\mathcal{A}}_{\xi}  s )  \bullet v + s \bullet  (
\nabla^{\mathcal{V}}_{\xi}  v ) .
        \end{align*}
\item For every tractor bundle $\mathcal V$, the value of the 
curvature $R^{\mathcal V}$ 
of the tractor connection $\nabla^{\mathcal V}$ is (for all vector fields
$\xi$, $\eta$ on $M$ and sections $v$ of $\mathcal VM$)
$$
R^{\mathcal V}(\xi,\eta)(v) = \kappa(\xi,\eta)\bullet v
,$$
where $\kappa\in\Omega^2(M,\mathcal AM)$ is the curvature of the Cartan
connection. 
    \end{enumerate}
\end{theorem}
 
 \begin{proof}
The first claim is obvious just by definition. The Lie bracket on the Lie
algebra is $\operatorname{Ad}$-equivariant.

The adjoint tractors are smooth equivariant functions $\mathcal G\to
\mathfrak g$. At the same time $\omega$ makes $T\mathcal G$ trivial. Now
all $\xi\in \mathcal G$ correspond to $\omega\circ \xi:\mathcal G\to \mathfrak
g$ and the right invariant fields $\xi$ correspond just to the adjoint
tractors. Since the Lie brackets of related fields is again related (here with respect to the
principal actions of the elements in $H$), the Lie bracket restricts to
$\mathcal X(G)^H$. Moreover, the right-invariant fields are projectable onto
vector fields on $M$, and the same argument applies to brackets of the
projections. 

The third claim also follows directly from the definitions. Indeed, writing
$\lambda$ for the representation $\lambda:H\to \operatorname{GL}(\mathbb V)$,
and $\lambda'$ for its differential at the unit, 
we recall
$\operatorname{exp}(t\operatorname{Ad}(g)(X))=g\operatorname{exp}(tX)g^{-1}$
and thus, differentiating we arrive at
$$
\lambda'(\operatorname{Ad}(g)(X))(\lambda(g)(v)) = \lambda(g)(\lambda'(X)(v))
.$$
Consequently, the bilinear map $\mathfrak g\times \mathbb V\to \mathbb V$
defined by $\lambda'$ is $G$ equivariant, it induces the map
$\bullet:\mathcal AM\times \mathcal VM\to \mathcal VM$ and the bracket formula
is just the defining property of a Lie algebra representation, in this
picture.

The next claim is a straightforward consequence of the fact that both $\{\
,\ \}$ and $\bullet$ are operations induced by $G$-equivariant maps. Thus we
may view them as living on the associated bundles to the extended $G$-principal
fiber bundle $\tilde{\mathcal G}$. The formulae are just simple properties
of the induced linear connections associated to a principal connection. 

The
same argument holds true for the last claim as well.
\end{proof}

Notice also the definition of the operation $\bullet$ extends to all natural
bundles $\mathcal V$, if we restrict the tractors only to the natural
subbundle $\operatorname{ker}\pi\subset \mathcal AM$ of all vertical right
invariant vector fields on $\mathcal G$, including the bracket compatibility 
property.
 
 \subsection{Fundamental derivative}

Consider the natural bundle  $\mathcal V : = \mathcal{G} \times_H \mathbb{V}$
associated to an $H$-representation $\lambda$ on $\mathbb{V}$. Then, viewing the adjoint
tractors as right invariant vector fields on $\mathcal G$, we can define the
differential operator $\operatorname{D} \colon \mathcal{AM} \times 
\mathcal VM \to \mathcal VM$ by the formula 
$$
\operatorname{D}_s v = s \cdot v,
$$ 
where $s \in \mathcal{AM}$ is any tractor in $\mathcal X(\mathcal
G)^H$ differentiating the
function $v:\mathcal G \to \mathbb V $. A simple check,
$$
     s ( u \cdot h ) \cdot v  = ( Tr^h\cdot s(u)  ) \cdot v 
 = s (u) \cdot (v \circ r^h ) = s(u) \cdot (\lambda_{h^{-1}} \cdot v ) =
\lambda_{h^{-1}}(s(u)\cdot v) 
,$$
reveals that the result is again a 
smooth $\mathbb{V}$-valued $H$-equivariant mapping on $\mathcal{G}$. We call
this operator $\operatorname{D}$ the \emph{fundamental derivative}. 

Notice that extending the tangent bundle to the adjoint tractors, we always
have a canonical way of `differentiating' on all natural bundles for all
Cartan geometries. As we may expect, there will be a lot of redundancy in
such differentiation, since the vertical tractors in the kernel of the
projection $\mathcal AM\to TM$ must act in an algebraic way due to the
equivariance of the functions $v$. 

Let us summarize some simple but very useful consequences of our
definitions: 

\begin{theorem}
\begin{enumerate} 
\item The fundamental derivative on the smooth functions (i.e.
we consider the trivial representation $\mathbb V=\mathbb R$) is just the
derivative in the direction of the projection:
$$
D_sf = \pi(s)\cdot f\, .
$$
\item If the adjoint tractor $s$ is vertical, i.e. $\pi(s)=0$, then for
every section $v$ of a natural bundle $\mathcal VM$,  
$$\operatorname{D}_s v = - s \bullet v
.$$
\item The fundamental derivative $\operatorname{D}$ is compatible with 
all natural operations on natural bundles (i.e. those coming from
$H$-invariant maps between the corresponding representation spaces). 
For example, having sections $v$, $v^*$, and $w$ of natural
bundles $\mathcal V$, $\mathcal V^*$, $\mathcal W$, and a function $f$
\begin{align*}
\operatorname{D}_s (f v) = (\pi (s) \cdot f ) v + f \operatorname{D}_s v 
\\
\operatorname{D}_s (v  \otimes w) = \operatorname{D}_s v  \otimes w + v \otimes 
\operatorname{D}_s w
\\
\pi(s)\cdot v^*(v) = (\operatorname{D}_s v^*)(v) + v^*(\operatorname{D}_s v) 
\end{align*}
\item If $\mathbb{V}$ is a $G$-representation, i.e. $\mathcal V$ is a tractor 
bundle, then 
\[ \nabla^{\mathcal{V}}_{\pi (s)} v = \operatorname{D}_s v + s \bullet v  
.\] 
\end{enumerate}
 \end{theorem}

\begin{proof}
The equivariant functions $\mathcal G\to \mathbb R$ are just the
compositions of functions $f$ on the base manifold $M$ with the projection
$p:\mathcal G\to M$. Thus the first property is obvious, $s\cdot (f\circ p) =
(Tp\cdot s)\cdot f = \pi(s)\cdot f$.

If $s$ is vertical, then $s(u)= \zeta_Z(u)$, where $\zeta_Z$
is a fundamental vector field given by $Z\in\mathfrak h$. 
Thus, 
$$
s(u)\cdot v = \frac d{dt}{}_{|0} 
r^{\operatorname{exp}tZ}(u)\cdot v = -\lambda'(Z)(v(u))=(-s\bullet v)(u)
.$$

The third property is again obvious -- as long as the natural operations
come from (multi)linear $H$-invariant maps, these will be compatible with
the differentiations of functions valued in those spaces, 
in the directions of the right-invariant vector fields.

In order to see the last formula, consider a vector $\xi\in
T_u{\mathcal G}\subset T_u\tilde{\mathcal G}$, covering a vector $\tau\in
T_xM$. 
Then the horizontal lift of $\tau$ at the frame $u\in\mathcal G\subset
\tilde{\mathcal G}$ is $\xi-\zeta_{\tilde\omega(\xi)} = \xi -
\zeta_{\omega(\xi)}$. But the tractor connection is defined as the
derivative of the equivariant function $v$ in any frame of $\tilde{\mathcal
G}$ in the direction of the horizontal lift and we obtain exactly the
requested formula interpreting $\xi$ as the value of the right-invariant
vector field $s$ (i.e. the adjoint tractor viewed as the equivariant function at
$u$ is expressed just via $\omega(\xi)$).   

\end{proof} 

If we leave the slot for the adjoint tractor in the fundamental derivative
free, we obtain the operator $\operatorname{D}:\mathcal VM \to \mathcal
A^*M\otimes \mathcal VM$, and this can be obviously iterated,
$$
\operatorname{D}^k:\mathcal VM \to \otimes^k\mathcal A^*M\otimes \mathcal VM
$$

Of course, there is a lot of redundancy in these higher order operators
compared to standard jet spaces of the sections. In the case of the first order, we can
identify the first jet prolongations $J^1\mathcal V$ of natural bundles as
the natural bundles associated to the representations $J^1\mathbb V$ which
are much smaller $H$-submodules in the modules 
$\mathbb V\oplus \operatorname{Hom}(\mathfrak g,\mathbb V)$
corresponding to the values of the fundamental derivative. This is a useful
observation because it implies that all invariant first order differential
operators on the homogeneous models extend naturally to the entire category
of the Cartan geometries with this model.

Before returning to the parabolic special cases, let us remark two more
facts. The proofs are using similar arguments as above and the 
reader can find them in \cite[sections 1.5.8, 1.5.9]{B}.

Expanding the formula for the exterior differential in the defining equation
of the curvature $\kappa$, we can express the differential bracket on
$\mathcal AM$:
\begin{equation}\label{bracket-formula-on-A}
\begin{aligned}
[s_1,s_2] &= \operatorname{D}_{s_1} s_2 - \operatorname{D}_{s_2}s_1 -
\kappa(\pi(s_1),\pi(s_2)) + \{s_1,s_2\}
\\
&= \nabla^{\mathcal A}_{\pi(s_1)}s_2 - \nabla^{\mathcal A}_{\pi(s_2)}s_1
-\kappa(\pi(s_1),\pi(s_2)) -\{s_1,s_2\} 
.\end{aligned}
\end{equation}

There is the generalization of the well known Bianchi identities for
curvature in the general Cartan geometry setting:
\begin{equation}\label{Bianchi1}
\sum_{\text{cyclic}}\left(\nabla^{\mathcal A}_{\xi_1}(\kappa(\xi_2,\xi_3)) -
\kappa([\xi_1,\xi_2],\xi_3) \right) = 0
\end{equation}
for all vector fields $\xi_1$, $\xi_2$, $\xi_3$, or its equivalent form for
triples of adjoint tractors:
\begin{equation}\label{Bianchi2}
\sum_{\text{cyclic}}\bigl( \{s_1,\kappa(s_2,s_3)\} - \kappa(\{s_1,s_2\},s_3)
+ \kappa(\kappa(s_1,s_2),s_3) + (\operatorname{D}_{s_1}\kappa)(s_2,s_3)
\bigr) = 0
.\end{equation}
Similarly, the Ricci
identity has got the general form for every section $v$ of a natural bundle
$\mathcal V$:
\begin{equation}\label{Ricci}
(\operatorname{D}^2 v)(s_1,s_2) - (\operatorname{D}^2 v)(s_2,s_1) =
-\operatorname{D}_{\kappa(s_1,s_2)}v + \operatorname{D}_{\{s_1,s_2\}}v
.\end{equation}

Notice, how easy we can read the classical identities for the affine
connections from the latter two. Since the Cartan geometry is modelled on
$\mathbb R^n=\operatorname{Aff}(n,\mathbb R)/\operatorname{GL}(n,\mathbb R)$
and the Lie algebra decomposes into direct sum of $\mathfrak{gl}(n,\mathbb
R)$-modules $\mathfrak g_{-1}=\mathbb R^n$ and $\mathfrak g_0 =
\mathfrak{gl}(n,\mathbb R)$, all the formulae decompose by homogeneities,
$\mathcal AM = TM\oplus P^1M$ (here $P^1M$ is the linear frame bundle of
$TM$), the bracket $\{\ ,\ \}$ becomes trivial on $TM$, while the mixed
bracket is just the evaluation. Thus, the Bianchi identity can be evaluated
on tangent vectors and it decomposes 
into the two
classical Bianchi identities for the torsion free connections, while it 
gets the
more complex quadratic form in general. Similarly for Ricci, evaluated on
$s_1$ and $s_2$ in $TM$. If the
torsion is zero, $\kappa$ has got only vertical values and thus the first
term on the right hand side is the algebraic action of the curvature (with plus sign),
while the other one vanishes.

\subsection{Back to parabolic geometries}

Recall the parabolic cases always come with the splitting 
 \[ \mathfrak{g} = \mathfrak{g}_{-} \oplus \mathfrak{p}, \] 
where $\mathfrak g_-$ is a subalgebra (but only a $\mathfrak g_0$ submodule). 
As before, we shall restrict ourselves to the $|1|$-graded case, although
            the below formulae easily extend to the general case.

%

Consider the category of parabolic geometries with the model $G/P$ and a
$P$-representation $\mathbb V$ which decomposes with respect to the action
of the grading element $E\in\mathfrak g_0$ into 
$\mathbb V=\mathbb V_0\oplus\dots\oplus \mathbb V_k$. The adjoint tractor
bundle has got the composition series
$$
\mathcal AM = TM \oplus \operatorname{End}TM\oplus T^*M
$$
where the middle term is a subbundle in $T^*M\otimes TM$ corresponding
to the group $G_0=P/P_+$. Again, $T^*M$ is the injecting part while $TM$ is the
projecting part, and the algebraic bracket $\{\ ,\ \}$ maps $T^*M\times TM\to
\operatorname{End}TM$. 

Once we fix a Weyl connection $\nabla$, 
the Rho-tensor becomes
a one-form valued in $T^*M\subset \mathcal AM$, we get the Rho-corrected
derivative $\nabla^\Rho$, all $P$-modules get split into $G_0$-irreducible
components which can be grouped according to the actions of the grading
element in $\mathfrak g_0$ etc.

\begin{theorem}\label{derivatives-parabolic} The fundamental derivative
$\operatorname{D}$ on $\mathcal V$ is given in terms of any choice of Weyl
connection by
$$
( \operatorname{D}_s v )_i  = ( \nabla^{\Rho}_{\pi (s)} v )_i - s_0\bullet
v_i - s_1\bullet v_{i-1}  
= \nabla_{\pi (s)} v_i - s_0 \bullet v_i + (\Rho(\pi (s)) - s_1) \bullet
v_{i-1}
$$
where $s=(\pi(s),s_0,s_1)$ and we indicate the splitting $\mathbb V=\mathbb
V_0\oplus\dots\oplus\mathbb V_k$ with respect to the action of the grading
element by the extra lower indices.

If $\mathcal V$ is a tractor bundle, then the tractor connection is given by
$$
(\nabla^{\mathcal V}_\xi v)_i = (\nabla^{\Rho}_\xi v)_i + \xi\bullet v_{i+1}
= \nabla_\xi v_i + \Rho(\xi)\bullet v_{i-1} + \xi\bullet v_{i+1}.
$$
\end{theorem}

\begin{proof}
Both formulae are direct consequences of the general formulae and the
definitions. The reader may also consult \cite[section 5.1.10]{B}.
\end{proof}

%



\subsection{Towards effective calculus for conformal geometry}
 
Now, with the general concepts and formulae at hand, it is obvious that
the Thomas' tractors come equipped with the nice tractor connection on all
conformal Riemannian manifolds in the sense of Cartan geometries and the
connection will be always given by the formulae in Theorem
\ref{Thomas-formula}, which are manifestly invariant. Moreover, we know
that the curvature of the Thomas' tractor connection on the sphere (with the
Maurer-Cartan form $\omega$) is zero.

But we still cannot be happy enough, for at least two reasons. First, we 
want to define the geometries by a structure on the tangent bundle 
and we shall come
to that question in the next lecture. Second, we need some more effective
manifestly natural operators than the fundamental derivative.

We shall only briefly comment on the latter problem and advise the readers
to look at \cite{A} for much more information.

Already Tracy Thomas constructed the differential operator $D$   
which is invariant for $\sigma \in \mathcal{E}[1]$, with values in $\mathcal
 T$ (we follow the usual convention of \cite{A} and write the projecting
part in the top, while the injecting part is in the bottom of the column
vector). We may follow our prolongation of the `conformal to Einstein'
equation from the second lecture. Starting with $\sigma$ in $\mathcal E[1]$,
we first put $\mu_a=\nabla_a\sigma$ in $\mathcal E_a[1]$ and then, contracting the equation
$\nabla_a\nabla_b\sigma+P_{ab}\sigma + g_{ab}\rho=0$ we see $-n\rho =
\nabla^a\nabla_a\sigma + P^a{}_a\sigma$. Thus, adjusting the $1/n$ factor, we
arrive at the operator $D:\mathcal E[1]\to \mathcal T$
\begin{align}\label{Thomas D-operator}
      \sigma \xrightarrow{D} \begin{pmatrix} n \sigma \\ n \nabla_a \sigma \\
-(\nabla^a\nabla_a + P^a{}_a ) \sigma  \end{pmatrix}.
  \end{align}
This \emph{Thomas' D-operator} extends to all densities $\mathcal E[w]$. 
 For $f \in \mathcal{E}[w]$ we define $Df$ in $\mathcal T[w-1]$ as 
\begin{align}
 Df =  \begin{pmatrix} (n+2w - 2)wf \\ (n+2w - 2) \nabla_a f \\
-(\nabla^a \nabla_a + w P^a{}_a ) f \end{pmatrix}     
. \end{align}

In particular, we should notice the following facts.  
For $w = 0$, the first nonzero slot in the column is $(n-2)\nabla_af$. Thus,
this operator must be invariant and we have recovered 
the usual differential of functions.

A much more interesting choice is $w = 1 - \frac{n}{2}$ since this kills the
first two components and the third one gets manifestly invariant. 
This way we get the second order operator  $\nabla^a\nabla_a + \frac{2-n}2 P^a{}_a $ and we recognize the celebrated Yamabe
operator mentioned already in the third lecture. (Just checking the
pedestrian way the invariance of this operator shows that the general theory
was worth the effort!)

This example indicates where the genuine tractor calculus goes with the aim
to construct manifestly invariant operators in an effective way. 


\section{The (co)homology and normalization}

We shall continue with parabolic $P \subset G$ and the Klein model $G \to
G/P$, mainly restricting to $|1|$-graded $\mathfrak g$. Thus $\mathfrak{g} = \mathfrak{g}_- \oplus \mathfrak{g}_0 \oplus \mathfrak{p}_+
= \mathfrak g_{-1}\oplus \mathfrak g_0\oplus \mathfrak g_1$. 

Recall that any choice of the reduction $\sigma:\mathcal G_0= \mathcal
G/\operatorname{exp}\mathfrak g_1\to \mathcal G$ 
of the structure group of a Cartan
geometry $(\mathcal G\to M,\omega)$ provides the pullback
$\sigma^*(\omega)$ which splits into the soldering form
$\theta\in\Omega^1(\mathcal G_0,\mathfrak g_{-1})$
(independent of the choice of $\sigma$), the Weyl connection $\nabla_a$, and
the Rho-tensor $\Rho_{ab}$, which is a $T^*M$ valued one-form on $M$.
Moreover, the adjoint tractor bundle splits as 
$$
\mathcal AM=TM\oplus \mathcal
A_0M \oplus T^*M
.$$ 

Our aim is now to find some suitable normalization allowing to construct a
natural Cartan connection from the data on $\mathcal G_0$. Once we succeed,
the tractor calculus related to this Cartan connection will become a natural
part of the geometry defined on $\mathcal G_0$. We shall see that
the crucial tool at our disposal is related to the cohomological properties 
of the Lie algebras in question. There are two equivalent ways: either to normalize the
curvature of the Cartan connection, or to normalize the curvature of a
suitable tractor connection. We shall show the first one, the other one 
was first achieved in \cite{Cap-Gover}, and both
are explained in full generality in \cite[chapter 3]{B}. 

\subsection{Deformations of Cartan connections}

The obvious idea is to quest for normalizations which will make the
curvatures of the Cartan connections as small as possible. In particular,
this will ensure that the right Cartan connections on homogeneous models
will be the Maurer-Cartan forms.

Consider two Cartan connections on the same principal bundle $\mathcal G\to
M$, $\omega$ and $\tilde\omega$. Then their difference $\Phi =
\tilde\omega - \omega$ clearly vanishes on all vertical vectors and is
right-invariant. Thus, we deal with a one-form $\Phi\in\Omega^1(M,\mathcal
AM)$. 

In the $|1|$-graded case, let us understand the `geometry' on $M$ as
the choice of the $G_0$-principal bundle $\mathcal G_0$ together with the
soldering form $\theta$, i.e. we adopt the most classical concept of a
G-structure as a reduction of the first order linear frame
bundle $P^1M$ to the structure group $G_0$. (We already mentioned in the
examples in lecture 3, that the
projective geometries are different.) It is obvious from our definitions
that the two Cartan connections will define the same structure in the latter
sense if and only if their difference has got values in $\mathfrak p$.
Thus, in our $|1|$-graded cases, $\Phi$ should be in $\Omega^1(M,\mathcal
A_0M\oplus T^*M)$.

In the general situation with longer gradings, we have to be much more
careful with the definition of the $G_0$-structure which has to be
generalized to the filtered manifolds. In brief, the tangent space inherits
the filtration by $\mathfrak p$-submodules of $\mathfrak g_-$ and a full
analog of the classical G-structure has to be considered on the associated
graded vector bundle $\operatorname{Gr}TM$. We shall not go to any details
here, the reader can find a detailed exposition in \cite[chapter 3]{B}. 

As we know, the curvature can be also viewed as the curvature function
$\kappa:\mathcal G \to \Lambda^2(\mathfrak g/\mathfrak p)^*\otimes \mathfrak
g$, and $(\mathfrak g/\mathfrak p)^*= \mathfrak p_+$ via the Killing form on
$\mathfrak g$. Thus, we should like to know how $\kappa$ changes if we
deform the Cartan connection by $\Phi$ in
$\Omega^1(M,\mathcal A_0M\oplus T^*M)$. 

Let us write $\kappa_\ell$ for the component of the curvature function of
homogeneity $\ell$, i.e. $\kappa_\ell\in \Lambda^2\mathfrak g_{-1}^*\otimes
\mathfrak g_{\ell-2}$ for the $|1|$-graded parabolic geometries.

\begin{lemma}\label{curvature-deformation}
Assume $\Phi\in\Omega^1(M,\mathcal A_0M\oplus T^*M)$ is of 
homogeneity $\ell=1$ or
$\ell=2$. Then the components of the curvature of homogeneities lower than
$\ell$ remain
unchanged, while the corresponding 
change of the $\mathfrak g_{-1}$ or $\mathfrak g_0$ component of the
curvature, viewed as function valued in 
$\Lambda^2\mathfrak g_{-1}^*\otimes \mathfrak g_i$ 
with $i=-1$ or $0$, respectively, is given by the formula
$$
(\tilde\kappa - \kappa)_i(X,Y) =  [X,\phi(Y)] - [Y,\phi(X)]
$$
where $\phi$ is the equivariant function $\mathcal G\to
\mathfrak g_{-1}^*\otimes(\mathfrak g_0\oplus\mathfrak g_1)$ representing
$\Phi$.
\end{lemma}

\begin{proof}
Considering vector fields $\xi$, $\eta\in T\mathcal G$,
$$
\tilde\omega(\xi) = \omega(\xi) + \phi(\omega(\xi))
.$$
Thus, hitting the equation with the exterior derivative, we obtain
$$
\operatorname{d}\tilde\omega(\xi,\eta) = \operatorname{d}\omega(\xi,\eta) +
\operatorname{d}\phi(\xi)(\omega(\eta)) -
\operatorname{d}\phi(\eta)(\omega(\xi))
+\phi(\operatorname{d}\omega(\xi,\eta)) 
,$$
while  
$$
[\tilde\omega(\xi),\tilde\omega(\eta)] = [\omega(\xi),\omega(\eta)] +
[\phi(\omega(\xi)),\omega(\eta)] + [\omega(\xi),\phi(\omega(\eta))] +
[\phi(\omega(\xi)),\phi(\omega(\eta))].
$$
Comparing the curvatures (as $\mathfrak g$-valued two forms on $\mathcal G$),
\begin{align*}
(\tilde\kappa-\kappa)(\xi,\eta) &= \operatorname{d}\phi(\xi)(\omega(\eta)) -
\operatorname{d}\phi(\omega(\eta))(\xi) + \phi(\operatorname{d}\omega(\xi,\eta))
\\
&\qquad - [\phi(\omega(\xi)),\omega(\eta)] + [\omega(\xi),\phi(\omega(\eta))]
+ [\phi(\omega(\xi)),\phi(\omega(\eta))]
.\end{align*}
Now, inspecting the homogeneities for $\phi$ valued in $\mathfrak g_i$ ($i=0$
corresponds to homogeneity $1$, while $i=1$ yields homogeneity $2$), the
first three terms will land in $\mathfrak g_i$, while the very last term is
either zero (if $i=1$) or sits in $\mathfrak g_i$ again (if $i=0$). Thus
only the two remaining brackets have got the values in $\mathfrak g_{i-1}$ 
and we obtain just the requested result if we write the vector fields as
functions on $\mathcal G$ with the help of $\omega$. 
\end{proof}
  
%
%
%

\subsection{Homology and cohomology}

The formula for the lowest homogeneity deformation of the curvature is a
special instance of a general algebraic construction, which 
works for arbitrary Lie algebra $\mathfrak{g}$ and $\mathfrak{g}$-module 
$\mathbb{V}$. We define the $k$-chains $C_k (\mathfrak{g}, \mathbb{V} )$ as
$$
C_k (\mathfrak{g}, \mathbb{V} ): = \Lambda^k \mathfrak{g} \otimes
\mathbb{V}
.$$ 
For each $k>0$ we define the linear operator $\delta_k \colon C_k \to C_{k-1}$
\begin{align*}
    \delta_k ( X_1 \wedge \dots \wedge X_k \otimes v) 
& = \sum_i (-1)^i \underbrace{X_1  \wedge \dots \wedge X_k}_{\text{omit $i$-th}} \otimes X_i \cdot v  \\ 
        & + \sum_{i < j} (-1)^{i+j} [X_i, X_j] \wedge 
\underbrace{ X_1  \wedge \dots \wedge X_k}_{\text{omit $i$-th, $j$-th}} \otimes
v\, .
\end{align*}
Then $\delta^2 = 0$ and thus $\delta$ acts on the chain complex 
$C (\mathfrak{g}, \mathbb{V})$ as a boundary operator. A direct check
reveals that $\delta$ is always a $\mathfrak g$-module homomorphism. 
Therefore we can define
the homology groups 
\[ 
H_k (\mathfrak{g}, \mathbb{V})  = \frac{\operatorname{ker} 
\delta_{k}}{\operatorname{im} \delta_{k+1}}
\] 
and they are again $\mathfrak{g}$-modules.

Note that $C_0 (\mathfrak{g}, \mathbb{V}) = \mathbb{V}$ 
and $C_1 (\mathfrak{g}, \mathbb{V}) \xrightarrow{\delta_1} 
C_0 (\mathfrak{g}, \mathbb{V})$ is given by $\delta_1 (X\otimes v) = 
X \cdot v$ which implies $H_0 (\mathfrak{g}, \mathbb{V}) = \mathbb{V}/ \langle X \cdot v 
\rangle = \mathbb{V} / \mathfrak{g} \cdot \mathbb{V}$.

In particular, considering the adjoint representation, 
$H_0(\mathfrak g,\mathfrak g) = \mathfrak g/[\mathfrak g,\mathfrak g]$.

Similarly to the homology, we can consider the dual construction 
for cochains 
$C^k (\mathfrak{g}, \mathbb{V})  = 
\Lambda^k \mathfrak{g}^* \otimes \mathbb{V}$ 
and coboundaries 
$\partial_k \colon C^k (\mathfrak{g}, \mathbb{V}) 
\to C^{k+1} (\mathfrak{g}, \mathbb{V})$ given by
\begin{align*}
    \partial_k \varphi ( X_0 , \dots , X_k \otimes v)   & = \sum_i (-1)^i X_i \varphi ( \underbrace{X_0 , \dots , X_k}_{\text{omit $i$-th}} ) \otimes X_i \cdot v  \\ 
        & + \sum_{i < j} (-1)^{i+j} \varphi ([X_i, X_j] ,\!\!\!\! 
\underbrace{X_1  , \dots , X_k}_{\text{omit $i$-th and $j$-th}}\!\!\!\!) \otimes v \ .
\end{align*}
Then $\partial$ provides a coboundary operator on the complex of cochains, 
i.e. $\partial^2 = 0$. The operators $\partial$ are again
$\mathfrak g$-module homomorphisms and we define the cohomology groups 
\[ 
H^k ( \mathfrak{g}, \mathbb{V})  =  \frac{\operatorname{ker} \partial_k}{\operatorname{im} \partial_{k-1}} .
\]

Again, the zero cohomology is easy to compute. Clearly $\partial_0(v)(X_0)=
X_0\cdot v$, while 
$$
\partial_1\psi(X,Y) = X\cdot\psi(Y) - Y\cdot\psi(X) -
\psi([X,Y]).
$$
Thus, $H^0 ( \mathfrak{g}, \mathbb{V}) = \mathbb{V}^{\mathfrak{g}} 
\subset \mathbb{V}$ is the kernel of the $\mathfrak{g}$-action. 
If we choose $\mathbb{V} = \mathfrak{g}$ with the adjoint action 
then $H^1 ( \mathfrak{g}, \mathfrak{g}) = \{\text{all derivatives} \} / \{ \text{inner derivatives} \} $.

Now, the crucial observation is that Lemma \ref{curvature-deformation}
expresses the lowest homogeneity of the deformation of the curvature of our
Cartan geometries, caused by $\phi\in\mathfrak
g_{-1}^*\otimes\mathfrak g_i$, via the coboundary 
differential $\partial \phi$ (the third
term is not there in our case since we deal with $|1|$-graded geometries).

%

For general parabolic geometries we also consider the curvature as an
equivariant function $\kappa:\mathcal G\to C^2(\mathfrak g_{-}\otimes
\mathfrak g)$ and $\mathfrak g$ is a $\mathfrak g_-$-module with the 
adjoint action. Even
in full generality, the Lemma \ref{curvature-deformation} holds true, i.e.
the lowest homogeneity of the curvature deformation caused 
by $\phi$ is given by
$\partial \phi$, see \cite[section 3.1.10]{B}.



\subsection{Normalization of parabolic geometries}

We should be interested in the cohomologies $H^k(\mathfrak
g_-,\mathfrak g)$, in particular in the second degree since the curvature
has got the values in the second degree cochains.
Recall Lemma \ref{curvature-deformation} which discussed how all possible 
deformations of the Cartan curvature (with positive homogeneities) impact
the curvature. In particular, we learned there that the available
deformation of the curvature fill the image of $\partial$ in
the second degree cochains (in the lowest non-trivial homogeneity). 

Now the crucial moment comes. Consider parabolic geometries with the
homogeneous model $G/P$, $\mathfrak g= \mathfrak g_{-}\oplus\mathfrak g_0
\oplus \mathfrak p_+$ and a $\mathfrak g$-module $\mathbb V$. Recall $(\mathfrak{g} / \mathfrak{p})^* \cong \mathfrak{g}_{-}^* \cong 
\mathfrak{p}_+$. Thus, the dual of
the space of cochains $C^k(\mathfrak g_-,\mathbb V)$ is $C^k(\mathfrak
p_+,\mathbb V^*)$ and there is the dual mapping $\partial^*:C^{k+1}(\mathfrak
p_+,\mathbb V^*)\to C^{k}(\mathfrak
p_+,\mathbb V^*)$. 
It was Kostant who noticed in his celebrated paper
\cite{Kost}, that there always is a scalar product 
$\langle \ , \ \rangle$ on the space of cochains 
$C^{k}(\mathfrak
p_+,\mathbb V^*)$ such that, identifying $C^k(\mathfrak p_+,\mathbb V^*)$ with
$C^k(\mathfrak g_-,\mathbb V)$, the latter dual map $\partial^*$ becomes the
adjoint operator to $\partial$. Moreover its formula then coincides with the
boundary operator $\delta$. We shall follow the (confusing) convention by many
authors and call this adjoint $\partial^*$ the \emph{codifferential}. 
In particular, $\partial^*$ is a $P$-module homomorphism.

Now,  we equivalently consider 
$$
H^k(\mathfrak{g}_-, \mathbb V) =
H^k(\mathfrak p_+,\mathbb V^*) = \frac{\operatorname{ker}\partial^* }
{\operatorname{im}\partial^*}
$$ 
and, applying the standard algebraic Hodge
theory, we get the decompositions (of $G_0$-modules)
\begin{equation}\label{Hodge-decomposition} 
C^k(\mathfrak{g}_-, \mathbb V) = \operatorname{im}\partial^* \oplus
\operatorname{ker}\partial = \operatorname{ker}\partial^* \oplus
\operatorname{im}\partial = \operatorname{im} \partial^* \oplus \operatorname{ker} \Box \oplus \operatorname{im} \partial ,
\end{equation}
where $\Box \equiv \partial \partial^* + \partial^* \partial$ (thus the
intersection of the kernels of $\partial$ and $\partial^*$). This means that
the cohomology $H^k(\mathfrak p_+,\mathbb V^*)=H^k(\mathfrak g_-,\mathbb V)$ 
equals to the kernel of the
algebraic Hodge Laplacian operator $\Box$. 

Further, we see that $\operatorname{ker}\partial^*$ is always the
complementary subspace to $\operatorname{im}\partial$ and in view of Lemma
\ref{curvature-deformation} we adopt the following normalization.

Notice $\partial^*$ is a $P$-module homomorphism and so it induces natural
transformations between the corresponding natural bundles. In particular, it
makes sense to apply $\partial^*$ to the curvatures of our Cartan
connections, i.e. there is the natural algebraic operator
$$
\partial^*: \Lambda^2T^*M\otimes\mathcal AM \to
T^*M\otimes \mathcal AM
$$
which preserves the homogeneities.

\begin{definition}
Let $(\mathcal G\to M, \omega)$ be a parabolic geometry with the homogeneous
model $G\to G/P$, $\mathfrak g = \mathfrak g_{-k}\oplus\cdots\oplus\mathfrak g_k$. 
We say that $\omega$ is a \emph{regular parabolic geometry}, if its
curvature $\kappa$ has got only positive homogeneities. The geometry is
called \emph{normal}, if its curvature is co-closed, i.e. $\partial^*\kappa=0$.

\end{definition}

Let us stress the following observation. 
The curvature of any normal parabolic geometry lies in the kernel of
$\partial^*$ and thus it projects to the natural bundle defined by
the cohomology $H^2(\mathfrak g_-,\mathfrak g)$. This is the so called
\emph{harmonic curvature} $\kappa^H \in \mathcal G\times_P H^2(\mathfrak
g_-,\mathfrak g)$.  

Let us restrict again our attention to $|1|$-graded geometries. First
notice, the regularity condition is empty in this case. Indeed, the
decomposition of $\kappa$ into its
homogeneity components coincides with the decomposition by its values, i.e.
values in $\mathfrak g_{i}$ are of homogeneity $i+2$, $i=-1$, $0$, $1$.

Further, there is a nice consequence of the Bianchi identity \eqref{Bianchi2}.
Consider the component $\kappa_i$ of the lowest homogeneity $\ell$. Then the
four terms in \eqref{Bianchi2} are of homogeneity at least, $\ell-1$,
$\ell-1$, $\ell$, $\ell$, respectively. But each homogeneity component in
\eqref{Bianchi2} has to vanish independently. Finally, the first two
terms represent exactly the differential $\partial \kappa$.   
      
We conclude that the lowest homogeneity non-zero component of the 
curvature should be
closed and thus, for normal geometries it must coincide with its harmonic
projection. Moreover, if all these harmonic components are zero, then we
conclude (by induction using the previous result) that the
entire curvature $\kappa$ must vanish, too. 
These results hold true even for
general parabolic geometries, the reader may
consult \cite[section 3.1.12]{B}.


Now we are ready to manage the normalization of the $|1|$-graded parabolic
geometries with $\mathfrak g=\mathfrak g_{-1}\oplus\mathfrak g_0\oplus
\mathfrak g_1$. Given any $G_0$-principal bundle $\mathcal G_0\to M$ with the
soldering form $\theta\in\Omega^1(M,\mathfrak g_{-1})$, i.e. a classical
$G_0$-structure, we consider the fiber bundle $\mathcal G=\mathcal
G_0\times \operatorname{exp}\mathfrak g_1$ and equip it with the obvious 
principal action of $P=G_0\ltimes \operatorname{exp}\mathfrak g_1$. 

If we choose any principal connection $\gamma$ on $\mathcal G_0$, then
$\theta\oplus \gamma$ is a Cartan connection on $\mathcal G_0\subset
\mathcal G$ and
choosing any $\Rho\in \Omega^1(M,T^*M)$, there is exactly one Cartan
connection $\omega$ on $\mathcal G$ coinciding with $\theta\oplus \gamma
\oplus \Rho$ on $T\mathcal G_0\subset T\mathcal G$. 

The connection is automatically regular and the lowest component of its
curvature can have homogeneity $1$. It is a simple exercise to see that
this component will coincide with the torsion $T$ of the connection $\gamma$
(e.g. viewed as the torsion part of the curvature of the Cartan connection
$\theta \oplus \gamma$). Moreover, changing the inclusion of $\mathcal
G_0\to \mathcal G$, i.e. choosing a Weyl connection for $\omega$, this
torsion part does not change at all.

We know that for the normal Cartan connections, this torsion has to coincide
with its harmonic part. Moreover, Lemma \ref{curvature-deformation} says
that we can modify the Cartan connection $\omega$ by a homogeneity one
deformation $\Phi$ so that this condition will be satisfied. 

In fact, this only recovers the very classical results about the
distinguished connections with special torsions on $G$-structures.  

For example, in conformal Riemannian geometry, there is no cohomology in
homogeneity one and thus we may always find torsion free connections. This
is, of course, no surprise since we may take any Levi-Civita connection of
one of the metrics in the class. But for the almost Grassmannian geometries
with $p\ge q\ge 3$, all the cohomology appears in homogeneity one only (with
two irreducible components) and
thus connections with torsions are unavoidable in general, unless we deal
with the homogeneous models. 

Next, we may assume that we have chosen the above connection $\gamma$ in
such a way, that its torsion is harmonic. In order to see the link between
the curvature of $\gamma$ and the curvature $\kappa$ of $\omega$, consider
the Cartan connection $\tilde\omega$ on $\mathcal G$ which would be given by
the choice $\Rho=0$. The Cartan connections $\theta\oplus \gamma$ and
$\tilde\omega$ are related by the inclusion $\mathcal G_0\to \mathcal G$ and
thus the curvature $\tilde \kappa$, restricted to $\mathcal G_0$ coincides
with the curvature $T+R$ of $\theta\oplus\gamma$. Thus, Lemma
\ref{curvature-deformation} says (with the deformation
$\Rho=\omega-\tilde\omega$) that the homogeneity two component of the
curvature of $\omega$ is
\[ 
\kappa_0 = R + \partial \Rho. 
\] 
Hitting this equality by $\partial^*$ gives 
\[ 
\partial^* \kappa_0 = \partial^* R + 
\partial^* \partial \Rho 
.\] 
But by homogeneity argument, $\partial^*\Rho$ would have values in
$\mathfrak g_2$ and thus vanishes automatically. Thus, the second term in
the latter equation equals $\Box \Rho$ and the normalization
condition will be satisfied if we choose $\Rho$ such that  
\begin{equation}\label{normal-Rho}
\Box \Rho  = - \partial^* R 
.\end{equation}
The final crucial observation is that the Laplacian acts by non-zero
constant multiples on all irreducible components, except the harmonic ones.
But we want to invert $\Box$ on $\operatorname{im}\partial^*$, which cannot
include any harmonic components. The final formula for $\Rho$ is
\begin{equation}\label{final-Rho}
\Rho = -\Box^{-1}\partial^*R
.\end{equation}

Summarizing, in order the construct the normal Cartan connection $\omega$ on
a manifold equipped with the relevant $G_0$-structure, we first choose
any connection $\gamma$ with harmonic torsion. Then we consider its
curvature $R$, apply the codifferential and compute the right coefficients
for each of its irreducible components. There are effective tools in the
representation theory allowing to compute them easily via the so called
Casimir operators. We have no space to go into details here.

Finally, there is the question about the uniqueness of our construction. The
answer is again hidden in cohomologies. If there are no positive homogeneity
components in $H^1(\mathfrak g_{-1},\mathfrak g)$, all our choices of the
deformations in both steps were unique. This is the case for nearly all
$|1|$-graded geometries. The only exceptions are the projective geometries
(and their complex versions), where we have to choose one of the connections
in the first step to define the structure. Then the Cartan connection is
already given uniquely via the next step in our construction.  

In the categorical language, there is the subcategory of the regular and
normal Cartan geometries, and this subcategory is equivalent to the category
of the infinitesimal $G_0$-structures on manifolds, up to some rare exceptions
due to the existence of positive homogeneities in first cohomologies in some
examples (where a similar equivalence exists, too). 

In conformal Riemannian geometry, i.e. $\mathfrak g=\mathfrak{so}(n+1,1)$,
there is no positive homogeneity first cohomology, while the entire second
cohomology is concentrated in homogeneity two (except of dimension $n=3$,
where it is homogeneity three). The operator  $\partial^*$ is just the
trace, so the image on the curvature of a Levi-Civita connection is the
Ricci tensor. The formula for $\Rho$ reflects the right choices of the
constants in the action of $\Box$, while the invariant Weyl part of the
curvature (shared by all Weyl connections) 
is $R + \partial \Rho$, the harmonic component in all
dimensions $n>3$. Of course, the geometry is locally isomorphic to the
conformal sphere if and only if this Weyl curvature vanishes.

We do not have space in this lecture to inspect further examples and
detailed computations. The readers may
look up many of them in \cite{B}, a few hundreds of pages of examples and
details for general parabolic geometries are there in chapters 3 through 5.
 

\section{The BGG machinery}

As well known, the linearized theories in Physics usually appear as locally
exact complexes of differential operators. A lot of attention was devoted to
this phenomenon in Mathematics, too. Already in the early days people around
Gelfand or Kostant knew that on the Klein models, the existence of such
complexes is an algebraic phenomenon related to homomorphisms of Verma
modules (which were understood as topological duals of the infinite jet
prolongations of the natural bundles), cf. \cite{BGG1,BGG2}.

The main message of this series of lectures is to show how remarkably
the algebraic features and phenomena from the Klein models extend to the
categories of Cartan geometries. The so called BGG machinery does exactly
this -- extends the complexes of the differential operators from the
homogeneous models to sequences on all Cartan geometries of the given type. 

In this last lecture we comment on this exciting development and we 
shall also come back to
the solutions of the `conformal to Einstein' equation \eqref{Einstein
condition} in terms
of constant tractors. On the way we shall touch the general construction 
of the latter sequences of operators and identify the equation \eqref{Einstein
condition} as one of the so called 1st BGG operators.

\subsection{The twisted de-Rham complexes}

Denote by $H^k_\mathbb{V} M$ the natural bundle associated to the $P$-module 
$H^k(\mathfrak{g}_-, \mathbb{V} ) $ of cohomologies with coefficients in a 
$G$-module $\mathbb{V}$. Notice, that by the Kostant's complete description
of the cohomologies, \cite{Kost}, the latter cohomology module is a $G_0$ module with
trivial action of $P_+$ and thus, it is completely reducible. 
In particular, $H^0_\mathbb{V} M$ is the bundle coming from the projecting
part of $\mathbb V$ which can be viewed as the orbit of the lowest
weight vector in $\mathbb{V}$ under the $\mathfrak{g}_0$-action. 
Our goal is to come to the following diagram of operators 
\begin{equation}
    \begin{tikzcd}\label{diagram1}
& \Omega^0(M, \mathcal{V}M ) 
\arrow{r}{\operatorname{d}_\mathbb{V}} \arrow[swap, xshift=-3pt]{d}{\pi} 
& \Omega^1(M, \mathcal{V}M ) \ar[swap, xshift=-3pt]{d}{\pi} 
\arrow{r}{\operatorname{d}_\mathbb{V}} & \dots  \\
    &  H^0_{\mathbb{V}} M \ar[dashed]{r}{\operatorname{D}} \ar[swap,xshift=3pt]{u}{\operatorname{L}} &  H^1_{\mathbb{V}} M \ar[swap, xshift=3pt]{u}{\operatorname{L}}  \ar[dashed]{r}{\operatorname{D}} & \dots
    \end{tikzcd}
\end{equation}    
where all the arrows have to be yet explained. As usual, we write $\mathcal{V}M$ 
for the tractor bundle over the manifold $M$ corresponding to
$\mathbb V$, and notice that $\partial^*$ is the adjoint of $\partial$ which
is a $P$-module homomorphism and thus, it gives rise to the natural algebraic
operator $\partial^*:\Omega^k(M, \mathcal{V}M)\to \Omega^{k-1}(M,\mathcal VM)$. 
Clearly, the projections $\pi$ are well defined only on the kernel of
$\partial^*$. We shall have to be careful about this.

The ideas presented below go back to \cite{Baston} and
\cite{BEastwood}, and they were further developed in \cite{C}.

Let us discuss the upper line in \eqref{diagram1} now.
First, restrict to the parabolic Klein model $G \to G/P$. Together with the
$G$-module $\mathbb{V}$, consider a $P$-module $\mathbb{W}$. 
Then there is the following identification of the sections of the 
tensor product bundle  $\mathcal{V}\otimes\mathcal{W}$. For any section
$s$ of
$\mathcal W$, i.e. an equivariant mapping $s:\mathcal G\to \mathbb W$, 
and $v\in\mathbb V$ consider the map
\[ 
s \otimes v \mapsto (\underbrace{g \mapsto s(g) \otimes g^{-1} 
\cdot v}_{\text{equivariant\ } G \to
\mathbb{W} \otimes \mathbb{V}} ), 
\]
which provides a natural isomorphism of the $G$-modules of sections 
\begin{equation}\label{tensor-iso}
\Gamma (\mathcal{W} ) \otimes \mathbb{V} 
\cong \Gamma ( \mathcal{W} \otimes \mathcal{V} )
.\end{equation}

Thus, if $F \colon \mathcal{W}_1 \to \mathcal{W}_2 $ is an arbitrary  
differential operator between the homogeneous vector bundles, 
then $F \otimes \operatorname{id}_{\mathbb{V}} = 
F_{\mathbb{V}} $ provides the \textit{twisted operator} $F_{\mathbb V} 
\colon \mathcal{W}_1\otimes \mathcal V \to \mathcal{W}_2 \otimes \mathcal
V$.

Considering the exterior differential $\operatorname{d}:
\Lambda^kT^*M\to \Lambda^{k+1}T^*M$,
this explains the whole first line in \eqref{diagram1}, at least on the
homogeneous model. On zero-degree forms, the exterior differential
is just the covariant derivative of the sections.
  
Let us look more
carefully on this example. At the level of first order jets, we can
express the twisted operator by means of the algebraic $P$-homomorphism
\begin{equation}\label{d-homo}
J^1 (\Lambda^k{\mathfrak{p}_+} \otimes \mathbb{V} ) \to 
\Lambda^{k+1}{\mathfrak{p}_+} \otimes \mathbb{V}, \quad
( f_0, Z \otimes f_1) \mapsto \partial f_0 + (k+1) Z \wedge f_1 
.\end{equation}

In general, if we write ${J}^r (\mathbb{W})$ and $\Bar{J}^r (\mathbb{W})$ 
for the standard fibers of the holonomic and semi-holonomic jet prolongations 
$J^r(\mathcal W)$, 
$\Bar J^r(\mathcal W)$,\footnote{We iterate the first jet prolongation.
Considering the first jets of sections of a bundle $\mathcal W$, the
jets in a fiber of $J^1(J^1\mathcal W)$ look in coordinates as 4-tuples
$(y^p,y^p_i,Y^p_j,Y^p_{ij})$ were $Y^p_{ij}$ do not need to be symmetric. 
These are the non-holonomic 2-jets. The
semi-holonomic ones remove part of the redundancy by requesting that the two
natural projections to 1-jets coincide, i.e. $y^p_i=Y^p_i$. This
construction extends to all orders and the semi-holonomic jets look in
coordinates nearly as the holonomic ones, just loosing the symmetry of the
derivatives. See e.g. \cite{KMS} for detailed exposition.} 
then the isomorphism \eqref{tensor-iso}
must hold true at the jet level, e.g. $\Bar{J}^r (\mathbb{W}) \otimes 
\mathbb{V} \cong \Bar{J} (\mathbb{W} \otimes \mathbb{V})$. 

Now the crucial observation comes: Although the jet prolongations
$J^r\mathcal W$ are no more natural bundles associated to $\mathcal G$ in
general, there is still no problem with the first jets. Thus, $J^1(\mathcal
W)=\mathcal G\times_P J^1(\mathbb W)$ and iterating this procedure, we
conclude that the semi-holonomic jet prolongations are natural bundles
again, i.e., $\bar J^r(\mathcal WM) = \mathcal G\times_P \bar J^r(\mathbb W)$
for the relevant $P$-module $\bar J^r(\mathbb W)$ (the standard fiber over
the origin in $G/P$ as the module with the action of the isotropy group $P$).
Moreover, we can construct a universal differential
operator $\mathcal WM\to \bar J^r(\mathcal WM)$ based on the iterated
fundamental derivative, which allows one to extend many invariant
operators from the homogeneous model to all Cartan geometries of this
type. 

Therefore, the so called \emph{strongly invariant operators}, i.e.
those coming from algebraic $P$-module homomorphisms $\bar J^r(\mathbb W_1)
\to \mathbb W_2$, enjoy a canonical extension to all Cartan geometries 
by means of the formulae obtained on the homogeneous model.

A careful
exposition of the algebraic structure of the semiholonomic jets and their
links to the strongly invariant operators can be found
in \cite{L}. 

This in particular applies for all first order operators and we are done
with the first line in \eqref{diagram1}, which is called the \emph{twisted
de-Rham} sequence. Obviously, there are many other ways for twisting the
de-Rham. For example, we could take the covariant exterior differential
$\operatorname{d}^{\omega}$ of
the tractor valued $k$-forms with respect to the tractor connection on
$\mathcal V$. A straightforward computation reveals
\begin{equation}\label{cov-dif}
\operatorname{d}^{\omega}{\varphi} = \operatorname{d}_{\mathbb{V}}\varphi + \iota_{\kappa_-} \varphi \ , 
\end{equation}
where $\kappa_-$ is the torsion part of the curvature $\kappa = \operatorname{d} \omega + \frac{1}{2} [ \omega, \omega]$. 

\subsection{BGG machinery}

Next, let us focus on the vertical arrows in \eqref{diagram1}. We already
know about the projections $\pi$, so we have to deal with $\operatorname{L}$'s.
\begin{equation*}
    \begin{tikzcd}
& \Omega^0(M, \mathbb{V} ) \arrow{r}{\operatorname{d}_\mathbb{V}} \arrow[swap, xshift=-3pt]{d}{\pi} & \Omega^1(M, \mathbb{V} ) \ar[swap, xshift=-3pt]{d}{\pi} \arrow{r}{\operatorname{d}_\mathbb{V}} & \dots  \\
    &  H^0_{\mathbb{V}} M  \ar[swap,xshift=3pt]{u}{\operatorname{L}} &  H^1_{\mathbb{V}} M \ar[swap, xshift=3pt]{u}{\operatorname{L}} 
    \end{tikzcd}
\end{equation*} 

The quite straightforward idea is to seek for differential operators
$\operatorname{L}$, such that $\operatorname{d}_{\mathbb V}\circ
\operatorname{L}$ are requested to be algebraically co-closed. Then the composition with the projection
$\pi$ makes sense and we could arrive at operators 
$\operatorname{D}$ 
 \[
 \begin{tikzcd}
  H^k_{\mathbb{V}} M \ar[dashed,
yshift=2pt]{rr}{\operatorname{D}=\pi\circ\operatorname{d}_{\mathbb V}\circ\operatorname{L}}
&{}&
H^{k+1}_{\mathbb{V}} M .
 \end{tikzcd}
\]

The most important (and demanding) step in the original
construction of the sequence of those operators in \eqref{diagram1} was the
following lemma in \cite{C}. Notice, \cite{I} suggests a different and more efficient
construction of these operators.

\begin{lemma}
On each irreducible component of $H^k_{\mathbb V}M$,  there is 
the unique strongly invariant operator 
$\operatorname{L}$ with values in $\operatorname{ker}\partial^*$ and 
splitting the projection $\pi$, 
\[
 \begin{tikzcd}
   H^k_{\mathbb{V}} M \ar[yshift=2pt]{r}{\operatorname{L}} & \Omega^k(M, \mathcal{V} ) \ar[yshift=-2pt]{l}{\pi} \ , 
 \end{tikzcd}
\]
such that $\operatorname{d}_{\operatorname{V} } \circ L \in 
\operatorname{ker} \partial^*$.
\end{lemma}

The proof in \cite{C} is very technical and there are many later
improvements in the literature, starting with \cite{I}.

The resulting sequence 
of operators
\[
 \begin{tikzcd}
 H^0_{\mathbb{V}} M \ar[dashed]{r}{\operatorname{D}_0} & H^1_{\mathbb{V}} M \ar[dashed]{r}{\operatorname{D}_1} & H^2_{\mathbb{V}} M \ar[dashed]{r}{\operatorname{D}_2} & \dots
 \end{tikzcd}
\]
is called the \emph{BGG sequence} associated with the tractor bundle
$\mathcal V$.

\begin{theorem}
For each $G$-module $\mathbb V$, the BGG sequence is well defined on each
Cartan geometry modelled on $G/P$ and it restricts to the celebrated 
BGG resolution on the homogeneous model. 

If the twisted de-Rham sequence on a Cartan geometry is a complex, 
then also the BGG sequence is a complex, and they both compute the same
cohomology of the underlying manifold.
\end{theorem}

A good example is the case when the Cartan geometry is torsion-free and the
curvature values act trivially on $\mathbb V$. Then the comparison
\eqref{cov-dif} of the twisted exterior differential and the covariant
exterior differential implies that the twisted de-Rham sequence will be
exact.

Often only a part of the whole BGG sequence is exact and many celebrated
complexes known in differential geometry can be recovered this way.

\subsection{The first BGG operators}
Finally, we are coming back to the first operators in BGG sequences. They
are always overdetermined operators
$\operatorname{D}:H^0_{\mathbb V}M\to H^1_{\mathbb V}M$. 
Moreover, by the very construction, its space of
solutions is in bijection with the space of the parallel tractors on the
homogeneous model. Unfortunately, this is not true in general and the so
called \emph{normal solutions} are those sections in the kernel of
$\operatorname{D}$ which correspond to parallel tractors. See \cite{H} for
interesting results on the normal solutions. Because of lack of
space in this last lecture, we shall just report briefly on the available
results. 

As carefully explained in \cite{D}, the normalization condition on the
canonical tractor connections can be written as $\partial^*(R^{\mathcal
V})=0$, 
considered on the space of 2-forms
valued in  endomorphisms
$\mathcal V\otimes \mathcal V^*$. 
At the same time, the normalization necessary for keeping the 1-1
correspondence between the solutions and the parallel tractors is rather 
$\partial^*_{\mathbb V}R^{\mathcal V}=0$, 
where the codifferential is modified, see \cite{D}.

So, although the values of our operator $\operatorname{L}$ on the
harmonic curvature are always algebraically co-closed, this is not enough.
  
%

The paper \cite{D} answers positively the question: 
Can we modify the Cartan connection so that 
$\partial^*_{\mathbb{V}} \circ \operatorname{d}_{\mathbb{V}} \circ \operatorname{L} (\kappa) = 0$
and thus the 1-1 correspondence will hold true for all Cartan geometries?

The first useful observation is the fact that the BGG machinery construction
survives without any changes if we restrict the deformations to the class of
connections: 
\[ 
C = \{ \tilde{\nabla} = \nabla + \Phi\ |\  \Phi \in \operatorname{ker} \partial^*_{\mathbb{V}} \otimes \operatorname{id} _{\mathbb{V}}, \Phi \text{\ has homogeneity $\geq 1$} \} 
\]

The main theorem of \cite{D} says:

\begin{theorem}
There is precisely one $\tilde{\nabla} \in C$ providing the $1-1$ 
correspondence between $\operatorname{ker} \operatorname{D}_0 $ and $\tilde{\nabla}$-parallel tractors.
\end{theorem}

At the very end, let us look again at the case of the `conformal to 
Einstein' equation \eqref{Einstein condition}, 
which is the first BGG operator for the choice of $\mathcal V$ equal to the
conformal standard tractors $\mathcal T$. 

Clearly, $H^0_{\mathbb T}M$ is the
projecting part $\mathcal E[1]$ of the tractors.
Further, a straightforward check reveals that the operator \eqref{Thomas
D-operator}
satisfies the conditions on $\operatorname{L}:H^0_{\mathbb T}M\to
\Omega^1(TM,\mathcal TM)$. Indeed, the entire space of zero-forms is 
in the kernel of $\partial^*$, the exterior derivative
$d^\omega$ is just the covariant derivative
\eqref{standard_tractor_connection} of the tractor, its projecting
slot vanishes, $\partial^*$ maps the injecting slot to zero by the
homogeneity, and $\partial^*$ is given by the trace in the middle slot, 
which vanishes, too. Since the geometry is torsion free, 
the exterior covariant derivative coincides with the twisted
derivative, see \eqref{cov-dif}. Finally, the projection of
$\operatorname{d}_{\mathbb T}\circ \operatorname{L}$ 
to the harmonic component provides just the
right operator \eqref{Einstein condition} on $\mathcal E[1]$.

In this very special case, there is no need to modify the tractor connection
in the above sense and thus there always is the 1-1 correspondence between
the solutions and the parallel tractors, which is again realized directly by
the operator $\operatorname{L}$. 

As already mentioned, many of important overdetermined operators appear as the first
BGG operators. A vast supply of interesting examples of the first order ones
appear in relation with the generalization of the classical problem of
metrizability of a projective geometry into the realm of filtered manifolds
and parabolic geometry. The projective case goes back to 19th century, the
generalization was recently worked out in \cite{K}.

\bigskip
\noindent\textbf{Acknowledgements.} The authors would like to thank the
organizers of the Summer School in Wis\l a for the hospitality and great
organization. The first author is also grateful to the audience for the
patience and good questions. 

The authors are most grateful to the anonymous
referee who pointed out several gaps in the original manuscript. 

The first author also acknowledges the support by the grant Nr. GA20-11473S
of the Czech Grant Agency. The second author was supported by the grant
MUNI/A/0885/2019 of the Masaryk University.

\end{document}